\documentclass{amsart}

\usepackage{graphicx,amsthm, amsmath, amssymb, mathrsfs}
\usepackage{mathtools}
\newtheorem{theorem}{Theorem}[section]

\newtheorem{prop}[theorem]{Proposition}

\newtheorem{lemma}[theorem]{Lemma}

\theoremstyle{definition}

\newtheorem{cor}[theorem]{Corollary}
\newtheorem{remark}[theorem]{Remark}

\numberwithin{equation}{section}

\def\Xint#1{\mathchoice
   {\XXint\displaystyle\textstyle{#1}}%
   {\XXint\textstyle\scriptstyle{#1}}%
   {\XXint\scriptstyle\scriptscriptstyle{#1}}%
   {\XXint\scriptscriptstyle\scriptscriptstyle{#1}}%
   \!\int}
\def\XXint#1#2#3{{\setbox0=\hbox{$#1{#2#3}{\int}$}
     \vcenter{\hbox{$#2#3$}}\kern-.5\wd0}}

\def\dashint{\Xint-}

\makeatletter
\def\moverlay{\mathpalette\mov@rlay}
\def\mov@rlay#1#2{\leavevmode\vtop{%
   \baselineskip\z@skip \lineskiplimit-\maxdimen
   \ialign{\hfil$\m@th#1##$\hfil\cr#2\crcr}}}
\newcommand{\charfusion}[3][\mathord]{
    #1{\ifx#1\mathop\vphantom{#2}\fi
        \mathpalette\mov@rlay{#2\cr#3}
      }
    \ifx#1\mathop\expandafter\displaylimits\fi}
\makeatother

\newcommand{\bigcupdot}{\charfusion[\mathop]{\bigcup}{\cdot}}
\newcommand{\R}{\mathbb R}
\newcommand{\al}{\alpha}
\newcommand{\M}{\mathcal{M}}

\begin{document}

\title[Bilinear estimates for general commutators]{Weighted estimates for bilinear fractional integral operators and their commutators}

\author{Cong Hoang}
\address{Cong Hoang, Department of Mathematics, University of Alabama, Tuscaloosa, AL 35487-0350}

\author{Kabe Moen}
\address{Kabe Moen, Department of Mathematics, University of Alabama, Tuscaloosa, AL 35487-0350}
%\curraddr{Department of Mathematics and Statistics,
%\email{xyz@math.university.edu}
\thanks{The second author was supported by NSF Grant \#1201504.}

\allowdisplaybreaks
%    General info
%\subjclass[2000]{Primary 54C40, 14E20; Secondary 46E25, 20C20}



\maketitle

\begin{abstract}
In this paper we will prove several weighted estimates for bilinear fractional integral operators and their commutators with BMO functions.  We also prove maximal function control theorem for these operators, that is, we prove the weighted $L^p$ norm is bounded by the weighted $L^p$ norm of a natural maximal operator when the weight belongs to $A_\infty$.  As a corollary we are able to obtain new weighted estimates for the bilinear maximal function associated to the bilinear Hilbert transform.

\hfill

\end{abstract}

%-------------------------------------------------
% NEW SECTION: INTRODUCTION
%-------------------------------------------------
\section{Introduction}

We are interested in the family of bilinear fractional integrals 
$$\mathsf{BI}_\al(f,g)(x)= \int_{\R^n}\frac{f(x-y)g(x+y)}{|y|^{n-\al}}\,dy, \qquad 0<\al<n.$$
The study of $\mathsf{BI}_{\alpha}$ was initiated by Kenig and Stein in \cite{C. E. Kenig and E. M. Stein} and Grafakos in \cite{L. Grafakos On multi} who proved that $\mathsf{BI}_{\alpha}:L^{p_1}\times L^{p_2}\rightarrow L^q$ when $1<p_1,p_2<\infty$ and $q$ satisfies $\frac1q=\frac{1}{p_1}+\frac{1}{p_2}-\frac{\al}{n}.$  The main interest of these operators is the singular nature of the kernel.  In fact, $\mathsf{BI}_\al$ has the same relationship to the bilinear Hilbert transform,
$$\mathsf{BH}(f,g)(x)=p.v.\int_{\mathbb{R}}\frac{f(x-y)\thinspace g(x+y)}{y}\thinspace dy$$
as the linear fractional integral has to the Hilbert transform.  We aim to study weighted norm inequalities of the form
$$
\mathsf{BI}_{\alpha}: L^{p_1}(v_1)\times L^{p_2}(v_2)\longrightarrow L^{q}(u).
$$
Weighted inequalities for these singular operators were unknown until the second author made some progress in  \cite{K. Moen bilinear} for the case when $p\leqslant q\leqslant 1$.   The main results of \cite{K. Moen bilinear} are stated in Theorem \ref{thmD} in the next section.

This paper was originally an attempt to expand the range for $p$ and $q$, but then as the theory developed, we were also interested in considering the effects of several types of commutators on $\mathsf{BI}_{\alpha}$. Given a linear operator $T$ and a function $b$, the commutator $[b,T]$ is defined to be
$$
[b,T]f=b \thinspace T(f)-T(bf).
$$
Coifman, Rochberg and Weiss \cite{R. R. Coifman and R. Rochberg and G. Weiss} introduced commutators of singular integral operators as a tool to extend the classical factorization theory of $H^p$ spaces. They proved that if $b\in BMO$ and $T$ is a singular integral operator, then $[b,T]$ in bounded on $L^p(\mathbb{R}^n)$ for $1<p<\infty$.

Weighted estimates for the linear fractional integral operator were done by Muckenhoupt and Wheeden \cite{B. Muckenhoupt and R.L. Wheeden} in the one weight case.  P\'erez \cite{CP} proved sufficient two weight bump conditions for the boundedness of $I_\al$.  The commutator of $I_\alpha$ was first considered by Chanillo \cite{S. Chanillo}, who showed that if $b\in BMO$, then $[b,I_\alpha]$ maps $L^p(\mathbb{R}^n)$ into $L^q(\mathbb{R}^n)$ with $\frac{1}{p}-\frac{1}{q}=\frac{\alpha}{n}$.  Weighted estimates for $[b,I_\al]$ were studied by D. Cruz-Uribe and the second author in \cite{D. Cruz-Uribe and K. Moen sharp norm for commutators} where it was shown that if $b \in BMO$, $1<p\leqslant q<\infty$ and $(u,v)$ is a pair of weights satisfying
$$
\sup_Q |Q|^{\frac{\alpha}{n}+\frac{1}{q}-\frac{1}{p}} \|u^\frac{1}{q}\|
_{A,Q}\| v^{-\frac{1}{p}}\|_{B,Q} < \infty
$$
where $A(t)=t^q\log(e+t)^{2q-1+\delta}$ and $B(t)=t^{p'}\log(e+t)^{2p'-1+\delta}$, we have
$$
\bigl\|[b,I_\alpha]\bigr\|_{L^q(u)} \lesssim \|b\|_{BMO} \|f\|_{L^p(v)}.
$$

%A dyadic version of this result and further applications were given by Lacey \cite{M. Lacey}.

%Weighted estimates for commutators of singular integral operators were also studied by Chung, Pereyra and P\'erez \cite{D. Chung and M. C. Pereyra and C. Perez} (see also \cite{D. Chung}). They proved that if $b \in BMO$, then
%$$
%\bigl\|[b,T]\bigr\|_{L^p(w)\rightarrow L^p(w)} \lesssim [w]_{A_p}^{2\max\left(1,\frac{p'}{p}\right)}
%$$
%where $w$ is an $A_p$ weight, i.e.
%$$
%[w]_{A_p}=\sup_Q\left(\dashint_Qw\right)\left(\dashint_Qw^{1-p'}\right)^{p-1} < \infty.
%$$

When considering a bilinear operator $\mathsf{BT}$, we define the commutators on the first and the second components to be
$$
[b,\mathsf{BT}]_1(f,g)=b \thinspace \mathsf{BT}(f,g)-\mathsf{BT}(bf,g)
$$
and
$$
[b,\mathsf{BT}]_2(f,g)=b \thinspace \mathsf{BT}(f,g)-\mathsf{BT}(f,bg).
$$
Let $\vec{b}=(b_1,...,b_N)$, where $b_i$'s are given functions, and $\vec{\beta}=(\beta_1,...,\beta_N)\in\ \{1,2\}^N$, the iterated product commutators of a bilinear operator $\mathsf{BT}$ is defined (from inner to outer) to be
$$
[\vec{b},\mathsf{BT}]_{\vec{\beta}}=[b_N,[b_{N-1}...,[b_2,[b_1,\mathsf{BT}]_{\beta_1}]_{\beta_2}...]_{\beta_{N-1}}]_{\beta_N}.
$$
In the linear case, these type of commutators were studied by P\'erez and Rivera-Rios \cite{C. Perez and Israel P. Rivera-Rios}. Given a bilinear operator $\mathsf{BT}$, we may rearrange the commutators in any order as the following Proposition states.

\begin{prop}
For any permutation $\sigma$ on $\{1,...,N\}$,
\begin{equation} \label{eq11}
[\sigma(\vec{b}),\mathsf{BT}]_{\sigma(\vec{\beta})}=[\vec{b},\mathsf{BT}]_{\vec{\beta}}
\end{equation}
where $\sigma(\vec{b})=(b_{\sigma(1)},...,b_{\sigma(N)})$ and $\sigma(\vec{\beta})=(\beta_{\sigma(1)},...,\beta_{\sigma(N)})$. In particular, equality \eqref{eq11} is valid for any permutation $\sigma_0$ be such that $\sigma_0(\vec{\beta})=(1,...,1,2,...,2)$.
\end{prop}

For simplicity in the notation and proof, from now on we will always assume that $\vec{\beta}=(1,...,1,2,...,2)$, and reserve the notation $m=m(\vec{\beta})$ to denote the number of 1's in $\vec{\beta}$.

%In this paper, we will study the commutators of the bilinear fractional integral operators
%$$
%\mathsf{BI}_\alpha(f,g)(x)=\int_{\mathbb{R}^n}\frac{f(x-t)g(x+t)}{|t|^{n-\alpha}}dt, \hspace{2mm} 0<\alpha<n.
%$$

%-------------------------------------------------
% NEW SECTION: MAIN RESULTS
%-------------------------------------------------
\section{Main Results}

Throughout this paper we will work with 2 different cases. The first case is when $p_1$ and $p_2$ get close enough to 1, which will force $p=\frac{p_1p_2}{p_1+p_2}\leqslant1$, while the second case is when $p>1$. Our departure is the following result of the second author \cite{K. Moen bilinear}.

\begin{theorem}[\cite{K. Moen bilinear}] \label{thmD}
Suppose $1<p_1,p_2$ and $q\leqslant1$ are such that $\frac{1}{2}<p=\frac{p_1p_2}{p_1+p_2}\leqslant q\leqslant1$. If $(u,v_1,v_2)$ are weights satisfying
\begin{equation*}
    \sup_Q |Q|^{\frac{\alpha}{n}+\frac{1}{q}-\frac{1}{p}} \left(\dashint_Qu^\frac{1}{1-q}\right)^\frac{1-q}{q}\bigl\|v_1^{-\frac{1}{p_1}}\bigr\|_{\phi_1,Q}\thinspace\bigl\|v_2^{-\frac{1}{p_2}}\bigr\|_{\phi_2,Q} < \infty
\end{equation*}
where $\phi_i(t)=t^{p_i'}\log(e+t)^{p_i'-1+\delta}$, $i\in\{1,2\}$, and $\left(\dashint_Qu^\frac{1}{1-q}\right)^{1-q}=\sup_Qu$ when $q=1$. Then, the inequality
$$
\bigl\|\mathsf{BI}_{\alpha}(f,g)\bigr\|_{L^q(u)} \lesssim \left\|f\right\|_{L^{p_1}(v_1)} \left\|g\right\|_{L^{p_2}(v_2)}
$$
holds for all $f\in L^{p_1}(v_1)$ and $g\in L^{p_2}(v_2)$.
\end{theorem}

By using a different technique, we are able to prove a similar result in the case when $1< p\leqslant q<\infty$.

\begin{theorem} \label{thmE}
Suppose $0<\alpha<n$, $p_1>r>1$, $p_2>s>1$,  $\frac{1}{r}+\frac{1}{s}=1$, $\frac{1}{p}=\frac{1}{p_1}+\frac{1}{p_2}$, $1<p\leqslant q < \infty$, and the set of weights $(u,v_1,v_2)$ satisfies
\begin{equation*}
    \sup_Q |Q|^{\frac{\alpha}{n}+\frac{1}{q}-\frac{1}{p}}\bigl\|u^\frac{1}{q}\bigr\|_{\psi,Q}\thinspace\bigl\|v_1^{-\frac{r}{p_1}}\bigr\|_{\phi_1,Q}^\frac{1}{r}\thinspace\bigl\|v_2^{-\frac{s}{p_2}}\bigr\|_{\phi_2,Q}^\frac{1}{s} < \infty
\end{equation*}
where $\psi,\phi_1,\phi_2$ are Young functions satisfying $\bar{\psi}\in B_{q'}$, $\bar{\phi}_1\in B_{\frac{p_1}{r}}$ and $\bar{\phi}_2\in B_{\frac{p_2}{s}}$. Then the inequality
$$
\bigl\|\mathsf{BI}_{\alpha}(f,g)\bigr\|_{L^q(u)} \lesssim \left\|f\right\|_{L^{p_1}(v_1)} \left\|g\right\|_{L^{p_2}(v_2)}
$$
holds for all $f\in L^{p_1}(v_1)$ and $g\in L^{p_2}(v_2)$.
\end{theorem}

For the general commutators defined on $\mathsf{BI}_\alpha$, we have Theorems \ref{thmA} and \ref{thmB} as stated below.

\begin{theorem} \label{thmA}
Suppose $0<\alpha<n$, $\vec{b}\in BMO^N$, $p_1>1$, $p_2>1$,  $\frac{1}{p}=\frac{1}{p_1}+\frac{1}{p_2}$, $\frac{1}{2}<p\leqslant q\leqslant1$, and the set of weights $(u,v_1,v_2)$ satisfies
\begin{equation*}
    \sup_Q |Q|^{\frac{\alpha}{n}+\frac{1}{q}-\frac{1}{p}}\bigl\|u^\frac{1}{1-q}\bigr\|_{\psi,Q}^\frac{1-q}{q}\thinspace\bigl\|v_1^{-\frac{1}{p_1}}\bigr\|_{\phi_1,Q}\thinspace\bigl\|v_2^{-\frac{1}{p_2}}\bigr\|_{\phi_2,Q} < \infty
\end{equation*}
where $\phi_1(t)=t^{p_1'}\log(e+t)^{(m+1)p_1'-1+\delta}$, $\phi_2(t)=t^{p_2'}\log(e+t)^{(N-m+1)p_2'-1+\delta}$, $\delta>0$, and $\psi(t)=t\log(e+t)^\frac{qN}{1-q}$ with $\bigl\|u^\frac{1}{1-q}\bigr\|_{\psi,Q}^{1-q}=\sup_Qu$ when $q=1$. Then, the inequality
$$
\bigl\|[\vec{b},\mathsf{BI}_{\alpha}]_{\vec{\beta}}(f,g)\bigr\|_{L^q(u)} \lesssim \|\vec{b}\|\left\|f\right\|_{L^{p_1}(v_1)} \left\|g\right\|_{L^{p_2}(v_2)}
$$
holds for all $f\in L^{p_1}(v_1)$ and $g\in L^{p_2}(v_2)$, where $\|\vec{b}\|=\prod_{i=1}^{N}\|b_i\|_{BMO}$.
\end{theorem}

\begin{theorem} \label{thmB}
Suppose $0<\alpha<n$, $\vec{b}\in BMO^N$, $p_1>r>1$, $p_2>s>1$,  $\frac{1}{r}+\frac{1}{s}=1$, $\frac{1}{p}=\frac{1}{p_1}+\frac{1}{p_2}$, $1<p\leqslant q < \infty$, and the set of weights $(u,v_1,v_2)$ satisfies
\begin{equation*}
    \sup_Q |Q|^{\frac{\alpha}{n}+\frac{1}{q}-\frac{1}{p}}\bigl\|u^\frac{1}{q}\bigr\|_{\psi,Q}\thinspace\bigl\|v_1^{-\frac{r}{p_1}}\bigr\|_{\phi_1,Q}^\frac{1}{r}\thinspace\bigl\|v_2^{-\frac{s}{p_2}}\bigr\|_{\phi_2,Q}^\frac{1}{s} < \infty
\end{equation*}
where
$$\phi_1(t)=t^{\left(\frac{p_1}{r}\right)'}\log(e+t)^{(mr+1)\left(\frac{p_1}{r}\right)'-1+\delta}
$$
$$
\phi_2(t)=t^{\left(\frac{p_2}{s}\right)'}\log(e+t)^{((N-m)s+1)\left(\frac{p_2}{s}\right)'-1+\delta}
$$
and $\psi(t)=t^q\log(e+t)^{(N+1)q-1+\delta}$, $\delta>0$. Then, the inequality
$$
\bigl\|[\vec{b},\mathsf{BI}_{\alpha}]_{\vec{\beta}}(f,g)\bigr\|_{L^q(u)} \lesssim \|\vec{b}\|\left\|f\right\|_{L^{p_1}(v_1)} \left\|g\right\|_{L^{p_2}(v_2)}
$$
holds for all $f\in L^{p_1}(v_1)$ and $g\in L^{p_2}(v_2)$.
\end{theorem}

We note here that the H\"older pairs $(r,s)$ in Theorems \ref{thmE} and \ref{thmB} exist and there are many such pairs. To name a few, we can start with $r=\frac{p_1}{p}$ and $s=\frac{p_2}{p}$, then we use the facts that $r<p_1$ and $s<p_2$ to obtain more choices by considering either $r=\frac{p_1}{p}+\epsilon$ or $s=\frac{p_2}{p}+\epsilon$, for small $\epsilon>0$. Also, we have completed just two thirds of the whole picture (i.e. the two cases: $p\leqslant q\leqslant1$ and $1<p\leqslant q$). The case when $p\leqslant1\leqslant q$ is still open.

While studying $\mathsf{BI}_\alpha$ and $[\vec{b},\mathsf{BI}_\alpha]_{\vec{\beta}}$, we need the following maximal operators: given two Young functions $\Phi$ and $\Psi$
$$\mathcal{M}^{\Phi,\Psi}_{\alpha}(f,g)(x)=\sup_{Q\ni x}|Q|^{\frac{\alpha}{n}}\|f\|_{\Phi,Q}\|g\|_{\Psi,Q}.$$
When $\al=0$ we write $\mathcal{M}^{\Phi,\Psi}_0=\M^{\Phi,\Psi}$.  When $\Phi(t)=t^r$ and $\Psi(t)=t^s$ we write 
$$\M^{\Phi,\Psi}_{\alpha}(f,g)(x)=\M_\al^{r,s}(f,g)(x)=\sup_{Q\ni x}|Q|^{\frac{\al}{n}}\left(\dashint_Q |f|^r\right)^{\frac1r}\left(\dashint_Q |g|^s\right)^{\frac1s}$$
and when $\Phi(t)=\Psi(t)=t$ we write
$$\M_\al(f,g)(x)=\M^{1,1}_\al(f,g)(x)=\sup_{Q\ni x}|Q|^{\frac{\al}{n}}\left(\dashint_Q |f|\right)\left(\dashint_Q |g|\right).$$
The controls that we have mentioned above are stated in the two theorems below.

\begin{theorem} \label{thmF}
Suppose $0<\alpha<n$, $0<q<\infty$ and $(r,s)$ is a H\"older pair. If the weight $w\in A_{\infty}$, then
$$
\int_{\mathbb{R}^n} \bigl|\mathsf{BI}_{\alpha}(f,g)(x)\bigr|^qw(x) \thinspace dx \lesssim \int_{\mathbb{R}^n} \M^{r,s}_{\alpha}(f,g)(x)^qw(x) \thinspace dx.
$$
\end{theorem}

\begin{theorem} \label{thmC}
Suppose $0<\alpha<n$, $0<q<\infty$ and $(r,s)$ is a H\"older pair. If the weight $w\in A_{\infty}$, then
$$
\int_{\mathbb{R}^n} \bigl|[\vec{b},\mathsf{BI}_{\alpha}]_{\vec{\beta}}(f,g)(x)\bigr|^qw(x) \thinspace dx \lesssim \|\vec{b}\|^q \int_{\mathbb{R}^n} \M^{\Phi,\Psi}_{\alpha}(f,g)(x)^qw(x) \thinspace dx
$$
where $\Phi(t)=t^r\log(e+t)^{mr}$ and $\Psi(t)=t^s\log(e+t)^{(N-m)s}$.
\end{theorem}

\hfill

In Theorem \ref{thmE}, if we consider special power-bump Young functions, then the condition on the weights $(u,v_1,v_2)$ become
\begin{equation} \label{eq21}
    \sup_Q |Q|^{\frac{\alpha}{n}+\frac{1}{q}-\frac{1}{p}}\left(\dashint_Q{u}\right)^{\frac{1}{q}} \left(\dashint_Q{v_1^{-\frac{r}{p_1-r}}}\right)^{\frac{p_1-r}{rp_1}} \left(\dashint_Q{v_2^{-\frac{s}{p_2-s}}}\right)^{\frac{p_2-s}{sp_2}} < \infty
\end{equation}
where $\left(\dashint_Q{v_1^{-\frac{r}{p_1-r}}}\right)^\frac{p_1-r}{r}=\left(\inf_Q{v_1}\right)^{-1}$ when $p_1=r$, and $\left(\dashint_Q{v_2^{-\frac{s}{p_2-s}}}\right)^\frac{p_2-s}{s}=\left(\inf_Q{v_2}\right)^{-1}$ when $p_2=s$.

It turns out that condition \eqref{eq21} can be characterized via the weak type and the strong type weighted boundedness of the maximal operator $\M^{r,s}_{\alpha}$, not only for $0<\alpha<n$ but also for $\alpha=0$. These results are stated in the following theorems.

\begin{theorem} \label{thmG}
Suppose $0\leqslant\alpha<n$, $p_1\geqslant r>1$, $p_2\geqslant s>1$, $1<p=\frac{p_1p_2}{p_1+p_2}\leqslant q$. Then $(u,v_1,v_2)$ satisfies condition \eqref{eq21} if and only if the inequality
$$
\sup_{\lambda>0}{\lambda \thinspace u\left(\{x:\M^{r,s}_{\alpha}(f,g)(x)>\lambda\}\right)^{\frac{1}{q}}} \lesssim \left\|f\right\|_{L^{p_1}(v_1)} \left\|g\right\|_{L^{p_2}(v_2)}
$$
holds for all $f\in L^{p_1}(v_1)$ and $g\in L^{p_2}(v_2)$.
\end{theorem}

\begin{theorem} \label{thmH}
Suppose $0\leqslant\alpha<n$, $p_1>r>1$, $p_2>s>1$, $1<p=\frac{p_1p_2}{p_1+p_2}\leqslant q$. If $(u,v_1,v_2)$ satisfies
\begin{equation} \label{eq22}
    \sup_Q |Q|^{\frac{\alpha}{n}+\frac{1}{q}-\frac{1}{p}}\bigl\|u^\frac{1}{q}\bigr\|_{\psi,Q}\thinspace\bigl\|v_1^{-\frac{r}{p_1}}\bigr\|_{\phi_1,Q}^\frac{1}{r}\thinspace\bigl\|v_2^{-\frac{s}{p_2}}\bigr\|_{\phi_2,Q}^\frac{1}{s} < \infty
\end{equation}
where $\psi,\phi_1,\phi_2$ are Young functions satisfying $\bar{\psi}\in B_{q'}$, $\bar{\phi}_1\in B_{\frac{p_1}{r}}$ and $\bar{\phi}_2\in B_{\frac{p_2}{s}}$, then the inequality
$$
\|\M^{r,s}_{\alpha}(f,g)\|_{L^q(u)} \lesssim \left\|f\right\|_{L^{p_1}(v_1)} \left\|g\right\|_{L^{p_2}(v_2)}
$$
holds for all $f\in L^{p_1}(v_1)$ and $g\in L^{p_2}(v_2)$.
\end{theorem}

So far, we have seen that the weak type weighted boundedness for $\M^{r,s}_{\alpha}$ is equivalent to condition \eqref{eq21}. Since strong type boundedness implies weak type one, it obviously implies condition \eqref{eq21}. In order to get the other way around, besides the stricter requirements that $p_1>r$ and $p_2>s$, we have to ``bump" up our condition on the weights by using the Orlicz norms as appeared in condition \eqref{eq22} of Theorem \ref{thmH}. However, things become much nicer in 1-weight settings [i.e. when $u^\frac{1}{q}=v_1^\frac{1}{p_1}v_2^\frac{1}{p_2}$].

\begin{theorem} \label{thmI}
Suppose $0\leqslant\alpha<n$, $p_1>r>1$, $p_2>s>1$, $\frac{1}{q}=\frac{1}{p}-\frac{\alpha}{n}$. Then the inequality
$$
\|\M^{r,s}_{\alpha}(f,g)\|_{L^q(w_1^\frac{q}{p_1}w_2^{\frac{q}{p_2}})} \lesssim \left\|f\right\|_{L^{p_1}(w_1)} \left\|g\right\|_{L^{p_2}(w_2)}
$$
holds if and only if the weights $(w_1,w_2)$ satisfy
\begin{equation}\label{onew}\sup_Q \left(\dashint_Q w_1^\frac{q}{p_1}w_2^{\frac{q}{p_2}}\right)^{\frac1q} \left(\dashint_Q{w_1^{-\frac{r}{p_1-r}}}\right)^{\frac{p_1-r}{rp_1}} \left(\dashint_Q{w_2^{-\frac{s}{p_2-s}}}\right)^{\frac{p_2-s}{sp_2}}<\infty.\end{equation}
\end{theorem}

Roughly speaking, when $p_1>r$ and $p_2>s$, in the multiple weight setting, $(u,v_1,v_2)$ we have
$$
\text{Strong bound for} \hspace{1mm} \M^{r,s}_\alpha \hspace{1mm} \Rightarrow \hspace{1mm} \text{Weak bound for} \hspace{1mm} \M^{r,s}_\alpha \hspace{1mm} \Leftrightarrow \hspace{1mm} \text{Condition \eqref{eq21}}
$$
and in vector weight setting $(w_1,w_2)$ and $u=w_1^{\frac{q}{p_1}}w_2^{\frac{q}{p_2}}$ we have
$$
\text{Strong bound for} \hspace{1mm} \M^{r,s}_\alpha \hspace{1mm} \Leftrightarrow \hspace{1mm} \text{Weak bound for} \hspace{1mm} \M^{r,s}_\alpha \hspace{1mm} \Leftrightarrow \hspace{1mm} \text{Condition\eqref{onew}}.
$$

As an immediate consequence of Theorems \ref{thmF} and \ref{thmI}, we have the following result.

\begin{cor}
Under the same assumptions as in Theorem \ref{thmI}, condition \eqref{onew} implies
$$
\|\mathsf{BI}_{\alpha}(f,g)\|_{L^q(w_1^\frac{q}{p_1}w_2^{\frac{q}{p_2}})} \lesssim \left\|f\right\|_{L^{p_1}(w_1)} \left\|g\right\|_{L^{p_2}(w_2)}
$$
for all $f\in L^{p_1}(w_1)$ and $g\in L^{p_2}(w_2)$.
\end{cor}

Finally, we end with an application of our estimates.  The associated maximal operator to the bilinear Hilbert transform is defined as
$$\mathsf{BM}(f,g)(x)=\sup_{r>0}\frac{1}{(2r)^n}\int_{[-r,r]^n}|f(x-y)g(x+y)|\,dy.$$
In the one dimensional case, this operator is studied in \cite{M. Lacey}, where it is shown that it satisfies
$$\mathsf{BM}:L^{p_1}(\R)\times L^{p_2}(\R)\rightarrow L^p(\R)$$
when $\frac1p=\frac1{p_1}+\frac1{p_2}$ and $p>2/3$.  Surprisingly, and contrary to the usual paradigm in harmonic analysis, the boundedness of the bilinear Hilbert transform was shown first and used to prove the boundedness of $\mathsf{BM}.$   Other than trivial conditions on the weights (i.e., assuming separate conditions on the weights such as {\it both} $w_1$ and $w_2$ belong to $A_p$ there are no known weighted estimates for $\mathsf{BM}$.  By H\"older's inequality we have that
$$\mathsf{BM}(f,g)(x)\leq \M^{r,s}(f,g)(x)$$
for any H\"older's pair of exponents $r$ and $s$ and therefore have the following corollaries.

\begin{cor} \label{BMtw}
Suppose $p_1>r>1$, $p_2>s>1$, $1<p=\frac{p_1p_2}{p_1+p_2}$. If $(u,v_1,v_2)$ satisfies
\begin{equation} 
    \sup_Q \bigl\|u^\frac{1}{p}\bigr\|_{\psi,Q}\thinspace\bigl\|v_1^{-\frac{r}{p_1}}\bigr\|_{\phi_1,Q}^\frac{1}{r}\thinspace\bigl\|v_2^{-\frac{s}{p_2}}\bigr\|_{\phi_2,Q}^\frac{1}{s} < \infty
\end{equation}
where $\psi,\phi_1,\phi_2$ are Young functions satisfying $\bar{\psi}\in B_{p'}$, $\bar{\phi}_1\in B_{\frac{p_1}{r}}$ and $\bar{\phi}_2\in B_{\frac{p_2}{s}}$, then the inequality
$$
\|\mathsf{BM}(f,g)\|_{L^p(u)} \lesssim \left\|f\right\|_{L^{p_1}(v_1)} \left\|g\right\|_{L^{p_2}(v_2)}
$$
holds for all $f\in L^{p_1}(v_1)$ and $g\in L^{p_2}(v_2)$.
\end{cor}

Finally we end with a one vector weight theorem.  In this case we will take the natural definition of $r=\frac{p_1}{p}$ and $s=\frac{p_2}{p}$.  

\begin{cor}
Suppose $p_1,p_2>1$, and $(w_1,w_2)$ are weights satisfying
\begin{equation} \label{onevecwp=q}
    \sup_Q\left(\dashint_Qw_1^\frac{p}{p_1}w_2^\frac{p}{p_2}\right)^\frac{1}{p} \left(\dashint_Qw_1^\frac{1}{1-p}\right)^\frac{p-1}{p_1} \left(\dashint_Qw_2^\frac{1}{1-p}\right)^\frac{p-1}{p_2} < \infty
\end{equation}
where $\left(\dashint_Qw_i^\frac{1}{1-p}\right)^{p-1}=(\inf_Qw_i)^{-1}$ when $p=1$, $i\in\{1,2\}$.\\
Then, $\mathsf{BM}$ is bounded from $L^{p_1}(w_1)\times L^{p_2}(w_2)$ to $L^{p,\infty}(w_1^{\frac{p}{p_1}}w_2^{\frac{p}{p_2}})$ whenever  $p\geqslant1$. \\
Moreover, $\mathsf{BM}$ is bounded from $L^{p_1}(w_1)\times L^{p_2}(w_2)$ to $L^{p}(w_1^{\frac{p}{p_1}}w_2^{\frac{p}{p_2}})$  whenever  $p>1$.
\end{cor}
In this paper, we will first give proof for Theorems \ref{thmA}, \ref{thmB} and \ref{thmC}. Theorems \ref{thmE} and \ref{thmF} can be proved using similar techniques as those in the proof of theorems \ref{thmB} and \ref{thmC}, respectively. The later Threorems are in fact easier than the former and follow from the exact techniques, so we choose to leave them for interested readers. We then will sketch the proof for Theorems \ref{thmG}, \ref{thmH} and \ref{thmI}, and end with an application of our results: a bilinear Stein-Weiss inequality.

\hfill

%-------------------------------------------------
% NEW SECTION: PRELIMINARIES
%-------------------------------------------------
\section{Preliminaries}
A dyadic grid $\mathscr{D}$ is a countable collection of cubes that satisfies the following properties:

\hspace{0.5cm} (1) $Q\in \mathscr{D} \Rightarrow \ell(Q)=2^k$ for some $k\in\mathbb{Z}$.

\hspace{0.5cm} (2) For each $k\in\mathbb{Z}$, the set $\{Q\in\mathscr{D}:\thinspace \ell(Q)=2^k\}$ forms a partition of $\mathbb{R}^n$.

\hspace{0.5cm} (3) $Q,P\in \mathscr{D} \Rightarrow Q\cap P \in \{\emptyset,P,Q\}$.

One very clear example for this concept is the dyadic grid that is formed by translating and then dilating the unit cube $[0,1)^n$ all over $\mathbb{R}^n$. More precisely, it is formulated as
$$
\mathscr{D}=\left\{2^{-k}\left([0,1)^n+m\right):\thinspace k\in\mathbb{Z},m\in \mathbb{Z}^n\right\}.
$$

In practice, we also make extensive use of the following family of dyadic grids.
$$
\mathscr{D}^t=\left\{2^{-k}\left([0,1)^n+m+(-1)^kt\right):\thinspace k\in\mathbb{Z},m\in \mathbb{Z}^n\right\}, \quad t\in\{0,1/3\}^n.
$$

Lerner \cite{A. K. Lerner} proved the following result.

\begin{theorem} \label{thm1}
Given any cube $Q$ in $\mathbb{R}^n$, there exists a $t\in \{0,1/3\}^n$ and a cube $Q_t\in\mathscr{D}^t$ such that $Q\subset Q_t$ and $\ell(Q_t)\leqslant 6\thinspace \ell(Q)$.
\end{theorem}

Next, we are going to give necessary details of Orlicz spaces. For more details, we refer the reader to \cite{D. Cruz-Uribe and J. M. Martell and C. Perez}. A Young function $\Phi:\thinspace [0,\infty)\rightarrow[0,\infty)$ is a continuous, convex and strictly increasing function with $\Phi(0)=0$ and $\frac{\Phi(t)}{t}\rightarrow\infty$ as $t\rightarrow\infty$. Given a Young function, there exists another Young function, denoted as $\bar{\Phi}$ and referred to as the associate function, that satisfies $t\leqslant\Phi^{-1}(t)\thinspace\bar{\Phi}^{-1}(t)\leqslant2t$ when $t>0$. For instance, the Young function $\Phi(t)=t^p$, $p>1$, has its associate Young function $\bar{\Phi}(t)=t^{p'}$ where $\frac{1}{p}+\frac{1}{p'}=1$. There are many more types of Young functions, but the most commonly seen are the ``log-bump" functions $\Phi(t)=t^r\log(e+t)^s$ for some $r>1$ and $s\in\mathbb{R}$.

The Orlicz average of $f$ over a cube $Q$ is given by
$$
\|f\|_{\Phi,Q}=\inf\left\{\lambda>0:\thinspace\dashint_{Q}\Phi\left(\frac{|f(x)|}{\lambda}\right)dx\leqslant1\right\}
$$
which is equivalent to
$$
\|f\|'_{\Phi,Q}=\inf_{\lambda>0}\left\{\lambda+\frac{\lambda}{|Q|}\int_Q\Phi\left(\frac{|f(x)|}{\lambda}\right)dx\right\}.
$$

This result is due to Krasnosel'ski\u{i} and Ruticki\u{i} \cite{M. A. Kranoselskii and J. B. Rutickii}. In fact,
$$
\|f\|_{\Phi,Q}\leqslant\|f\|'_{\Phi,Q}\leqslant2\|f\|_{\Phi,Q}.
$$

The Orlicz maximal function is then defined to be
$$
M_{\Phi}(f)(x)=\sup_{Q\ni x}\|f\|_{\Phi,Q}.
$$

P\'erez \cite{C. Perez} gave a necessary and sufficient condition for the boundedness of these Orlicz maximal operators.

\begin{theorem} \label{thm3}
For any $p\in(1,\infty)$,
$$
\|M_{\Phi}f\|_{L^p(\mathbb{R}^n)}\leqslant C\thinspace \|f\|_{L^p(\mathbb{R}^n)}
$$
if and only if $\thinspace\Phi$ satisfies the $B_p$ integrability condition, i.e. there exists $c>0$ such that
$$
\int_c^{\infty}\frac{\Phi(t)}{t^{p+1}}\thinspace dt < \infty.
$$
\end{theorem}

There is also a generalized H\"older inequality for these Orlicz averages.

\begin{lemma} \label{thm4}
If $\Phi,\Psi,\Theta$ are Young functions such that
$$
\Phi^{-1}(t)\Psi^{-1}(t)\lesssim \Theta^{-1}(t), \hspace{2mm} \forall t\geqslant t_0\geqslant0
$$
then
$$
\|fg\|_{\Theta,Q}\lesssim \|f\|_{\Phi,Q} \|g\|_{\Psi,Q}.
$$
In particular, for any Young function $\psi$,
$$
\dashint_{Q}{|f(x)\thinspace g(x)|\thinspace dx}\leqslant2\thinspace\|f\|_{\psi,Q}\thinspace\|g\|_{\bar{\psi},Q}.
$$
\end{lemma}

When $p>1$, a weight $w\in A_p$ if and only if
$$
\sup_Q\left(\dashint_Qw\right)\left(\dashint_Qw^{1-p'}\right)^{p-1}<\infty.
$$
When $p=1$, we have $w\in A_1$ if and only if
$$
Mw(x)\leqslant C\thinspace w(x), \qquad a.e. \thinspace x\in \mathbb{R}^n
$$
where $M$ is the Hardy-Littlewood maximal function.  Finally we define $A_\infty$ as the union of all $A_p$ classes for $p>1$.   Also from \cite{J. Duoandikoetxea} we know the following facts.

\begin{lemma} \label{thm10}
If $w\in A_\infty$ then the following hold:

i) for every $\eta\in(0,1)$, there exists $\kappa\in(0,1)$ such that: given a cube $Q$ and

\hfill $S\subseteq Q$ with $|S|\leqslant\eta\thinspace |Q|$, we will also have $w(S)\leqslant\kappa\thinspace w(Q)$;

ii)  there exist an $m>1$ such that
$$
    \left(\dashint_Qw^{m}\right)^{\frac{1}{m}} \leqslant C \thinspace \dashint_Qw.
$$
\end{lemma}

Next, we would like to briefly discuss the bilinear Muckenhoupt condition, $A_{\vec{P},q}$ condition, which was introduced by the second author in \cite{K. Moen multilinear}. A set of weights $(w_1,...,w_m)$ is said to be in the class $A_{\vec{P},q}$ if
$$
\sup_Q \left(\dashint_Q (w_1w_2)^q\right)^{\frac{1}{q}}\left(\dashint_Q{w_1^{-p_1'}}\right)^{\frac{1}{p_1'}}\left(\dashint_Q{w_2^{-p_2'}}\right)^{\frac{1}{p_2'}} <\infty.
$$
The second author  also proved that if $p_i\leqslant q_i$ and $\frac{1}{q}=\frac{1}{q_1}+\frac{1}{q_2}$, then
$$
\bigcup_{q_1,q_2}(A_{p_1,q_1}\times A_{p_2,q_2})\subsetneq A_{\vec{P},q}
$$
where the inclusion was shown to be strict.  \begin{theorem} \label{thm5}
Suppose $1<p_1,p_2<\infty$, and $(w_1,w_2)\in A_{\vec{P},q}$, then we have
$$
(w_1w_2)^q\in A_{2q} \qquad \text{and} \qquad w_i^{-p_i'}\in A_{2p_i'}.
$$
\end{theorem}

Cruz-Uribe, Martell and P\'erez \cite{D. Cruz-Uribe and J. M. Martell and C. Perez} proved an extrapolation theorem for $A_\infty$ weights. Namely,

\begin{theorem} \label{thm6}
Suppose there exist $p_0\in (0,\infty)$ such that
$$
\int_{\mathbb{R}^n}|f|^{p_0}w\leqslant C \int_{\mathbb{R}^n}|g|^{p_0}w \qquad \forall w\in A_\infty
$$
then we have
$$
\int_{\mathbb{R}^n}|f|^pw\leqslant C \int_{\mathbb{R}^n}|g|^pw \qquad \forall w\in A_\infty, \forall p\in(0,\infty).
$$
\end{theorem}

Finally, we will need the concept of bounded mean oscillation. Let $BMO$ denote the space of functions of bounded mean oscillation, i.e., functions $b$ such that
$$
\|b\|_{BMO}=\sup_Q\dashint_Q|b(x)-b_Q| \thinspace dx < \infty
$$
where $b_Q=\dashint_Qb(x) \thinspace dx$.

$BMO$ functions satisfy the exponential integrability which is a consequence of the John-Nirenberg theorem.

\begin{theorem} \label{thm7}
Given $b\in BMO$, there exists a constant $c_n$ such that for every cube Q,
$$
\sup_Q \dashint_Q\exp \left(\frac{|b(x)-b_Q|}{2^{n+2}\|b\|_{BMO}}\right)dx \leqslant c_n.
$$
In particular,
$$
\|b-b_Q\|_{\exp L,Q}\leqslant c_n 2^{n+2} \|b\|_{BMO}.
$$
\end{theorem}
A proof of Theorem \ref{thm7} can be found in \cite{J. L. Journe}.

\begin{cor} \label{thm8}
If $b\in BMO$, then for any $\xi>0$,
$$
\bigl\||b-b_Q|^\xi\bigr\|^\frac{1}{\xi}_{\exp(L^\frac{1}{\xi}),Q}\lesssim c_n 2^{n+2} \|b\|_{BMO}
$$
where $\exp(L^\frac{1}{\xi})$ stands for the Young function $\psi(t)\approx \exp(t^\frac{1}{\xi})-1$.
\end{cor}
\begin{proof}
By definition, we have
\begin{equation*}
\begin{split}
    \bigl\||b-b_Q|^\xi\bigr\|_{\exp(L^\frac{1}{\xi}),Q}
    & = \inf \left\{\lambda>0: \hspace{1mm}\frac{1}{|Q|}\int_Q\left[\exp\left(\frac{|b(x)-b_Q|}{\lambda^\frac{1}{\xi}}\right)-1\right] dx \leqslant 1\right\} \\
    & = \inf \left\{\lambda^{\xi}>0: \hspace{1mm}\frac{1}{|Q|}\int_Q\left[\exp\left(\frac{|b(x)-b_Q|}{\lambda}\right)-1\right] dx \leqslant 1\right\} \\
    & = \bigl\|b-b_Q\bigr\|_{\exp L,Q}^\xi
\end{split}
\end{equation*}
which implies the desired estimate.
\end{proof}

Through out this paper, we will make extensive use of the following proposition, which is actually a discrete H\"older inequality.

\begin{prop} \label{thm9}
Suppose $p_1,p_2>1$, $p_3>0$, and  $\frac{1}{p}=\frac{1}{p_1}+\frac{1}{p_2}<1\leqslant\frac{1}{p_1}+\frac{1}{p_2}+\frac{1}{p_3}$. We have the following inequality for non-negative sequences $\{a_j\}$, $\{b_j\}$, and $\{c_j\}$
$$
\sum_j a_jb_jc_j\leqslant \left(\sum_ja_j^{p_1}\right)^{\frac{1}{p_1}} \left(\sum_jb_j^{p_2}\right)^{\frac{1}{p_2}} \left(\sum_jc_j^{p_3}\right)^{\frac{1}{p_3}}.
$$
\end{prop}

\hfill

%-------------------------------------------------
% NEW SECTION: PROOF OF THEOREM 2.3
%-------------------------------------------------
\section{Proof of Theorem \ref{thmA}}

Without loss of generality, we may assume that $f$ and $g$ are non-negative, bounded and compactly supported. By induction, we can prove that

\begin{equation} \label{eq41}
\begin{split}
    [\vec{b}, & \mathsf{BI}_\alpha]_{\vec{\beta}}(f,g)(x) \\
    = & \int_{\mathbb{R}^n} \prod_{i=1}^{m}\bigl(b_i(x)-b_i(x-y)\bigr) \prod_{i=m+1}^{N}\bigl(b_i(x)-b_i(x+y)\bigr) \frac{f(x-y)g(x+y)}{|y|^{n-\alpha}} dy.
\end{split}
\end{equation}
For each $Q\in \mathscr{D}$, let $\lambda_i=\lambda_i(Q)=\dashint_{3Q}b_i(x)dx$ where $i=1,...,N$, we have

\begin{equation*}
\begin{split}
    \prod_{i=1}^{m}\bigl(b_i(x)-b_i(x-y)\bigr) & = \prod_{i=1}^{m}\Bigl[\bigl(b_i(x)-\lambda_i\bigr)+\bigl(\lambda_i-b_i(x-y)\bigr)\Bigr] \\
    & = \sum_{A\subseteq\{1,...,m\}}\prod_{i\in A}\bigl(b_i(x)-\lambda_i\bigr) \prod_{i\in \bar{A}}\bigl(\lambda_i-b_i(x-y)\bigr)
\end{split}
\end{equation*}
and similarly,
$$
\prod_{i=m+1}^{N}\bigl(b_i(x)-b_i(x+y)\bigr)=\sum_{B\subseteq\{m+1,...,N\}}\prod_{i\in B}\bigl(b_i(x)-\lambda_i\bigr) \prod_{i\in \bar{B}}\bigl(\lambda_i-b_i(x+y)\bigr).
$$
Hence
\begin{equation*}
\begin{split}
    & \prod_{i=1}^{m}\bigl(b_i(x)-b_i(x-y)\bigr) \prod_{i=m+1}^{N}\bigl(b_i(x)-b_i(x+y)\bigr) = \\
    & \sum_{A\subseteq\{1,...,m\}}\sum_{B\subseteq\{m+1,...,N\}} \prod_{i\in A\cup B}\bigl(b_i(x)-\lambda_i\bigr) \prod_{i\in \bar{A}}\bigl(\lambda_i-b_i(x-y)\bigr) \prod_{i\in \bar{B}}\bigl(\lambda_i-b_i(x+y)\bigr).
\end{split}
\end{equation*}
This estimate together with \eqref{eq41} yield

\begin{equation} \label{eq42}
\begin{split}
    & \bigl|[\vec{b}, \mathsf{BI}_\alpha]_{\vec{\beta}}(f,g)(x)\bigr| \\
    & \leqslant \sum_{A\subseteq\{1,...,m\}}\sum_{B\subseteq\{m+1,...,N\}} \int_{\mathbb{R}^n} \prod_{i\in A\cup B}|b_i(x)-\lambda_i| \prod_{i\in \bar{A}}|b_i(x-y)-\lambda_i| \\
    & \hspace{6.7cm} \prod_{i\in \bar{B}}|b_i(x+y)-\lambda_i| \frac{f(x-y)g(x+y)}{|y|^{n-\alpha}} dy \\
    & \lesssim \sum_{A\subseteq\{1,...,m\}}\sum_{B\subseteq\{m+1,...,N\}} \sum_{Q\in\mathscr{D}} |Q|^{\frac{\alpha}{n}-1}\int_{|y|_\infty\leqslant\ell(Q)} \prod_{i\in A\cup B}|b_i(x)-\lambda_i| \\
    & \hspace{2.7cm} \prod_{i\in \bar{A}}|b_i(x-y)-\lambda_i| \prod_{i\in \bar{B}}|b_i(x+y)-\lambda_i| f(x-y)g(x+y)\thinspace dy \hspace{1mm} \chi_Q(x). \\
\end{split}
\end{equation}

\noindent Since $q\leqslant1$, we have

\begin{equation*}
\begin{split}
    & \int_{\mathbb{R}^n}\bigl|[\vec{b}, \mathsf{BI}_\alpha]_{\vec{\beta}}(f,g)(x)\bigr|^q u(x) \thinspace dx \\
    & \lesssim \sum_{A\subseteq\{1,...,m\}}\sum_{B\subseteq\{m+1,...,N\}} \sum_{Q\in\mathscr{D}} |Q|^{\left(\frac{\alpha}{n}-1\right)q} \\
    & \hspace{1cm} \int_Q\left[\int_{|y|_\infty\leqslant\ell(Q)} \prod_{i\in \bar{A}}|b_i(x-y)-\lambda_i| \prod_{i\in \bar{B}}|b_i(x+y)-\lambda_i| f(x-y)g(x+y)\thinspace dy\right]^q \\
    & \hspace{8cm} \left[\prod_{i\in A\cup B}|b_i(x)-\lambda_i|\right]^q u(x) \thinspace dx.
\end{split}
\end{equation*}

\noindent If we use H\"older inequality with the pair $\left(\frac{1}{q},\frac{1}{1-q}\right)$, we will arrive at the inequality

\begin{equation*}
\begin{split}
    & \int_{\mathbb{R}^n}\bigl|[\vec{b}, \mathsf{BI}_\alpha]_{\vec{\beta}}(f,g)(x)\bigr|^q u(x) \thinspace dx \\
    & \lesssim \sum_{A\subseteq\{1,...,m\}}\sum_{B\subseteq\{m+1,...,N\}} \sum_{Q\in\mathscr{D}} |Q|^{\left(\frac{\alpha}{n}-1\right)q} \\
    & \hspace{0.7cm} \left[\int_Q\int_{|y|_\infty\leqslant\ell(Q)} \prod_{i\in \bar{A}}|b_i(x-y)-\lambda_i| \prod_{i\in \bar{B}}|b_i(x+y)-\lambda_i| f(x-y)g(x+y)\thinspace dydx\right]^q \\
    & \hspace{6.2cm} \left[\int_Q \prod_{i\in A\cup B}|b_i(x)-\lambda_i|^\frac{q}{1-q} u(x)^\frac{1}{1-q} \thinspace dx\right]^{1-q}.
\end{split}
\end{equation*}

\noindent By a change of variables, we have

\begin{equation} \label{eq43}
\begin{split}
    & \int_{\mathbb{R}^n}\bigl|[\vec{b}, \mathsf{BI}_\alpha]_{\vec{\beta}}(f,g)(x)\bigr|^q u(x) \thinspace dx \\
    & \lesssim \sum_{A\subseteq\{1,...,m\}}\sum_{B\subseteq\{m+1,...,N\}} \sum_{Q\in\mathscr{D}} |Q|^{\frac{\alpha}{n}q+1} \\
    & \hspace{4.1cm} \left[\dashint_{3Q} \prod_{i\in \bar{A}}|b_i(t)-\lambda_i|f(t)\thinspace dt \hspace{1mm} \dashint_{3Q} \prod_{i\in \bar{B}}|b_i(z)-\lambda_i|g(z)\thinspace dz\right]^q \\
    & \hspace{6cm} \left[\dashint_{3Q} \prod_{i\in A\cup B}|b_i(x)-\lambda_i|^\frac{q}{1-q} u(x)^\frac{1}{1-q} \thinspace dx\right]^{1-q}.
\end{split}
\end{equation}
Now we use the generalized H\"older inequality, Theorem \ref{thm7} and Corollary \ref{thm8} to obtain the following estimates:
\begin{equation*}
\begin{split}
    \dashint_{3Q} \prod_{i\in \bar{A}}|b_i(t)-\lambda_i|f(t)\thinspace dt
    & \lesssim \prod_{i\in \bar{A}} \|b_i-\lambda_i\|_{\exp L,3Q} \hspace{1mm}\|f\|_{L(\log L)^{|\bar{A}|},3Q} \\
    & \lesssim \prod_{i\in \bar{A}} \|b_i\|_{BMO} \hspace{1mm} \|f\|_{L(\log L)^{|\bar{A}|},3Q}
\end{split}
\end{equation*}
and
\begin{equation*}
\begin{split}
    \dashint_{3Q} \prod_{i\in \bar{B}}|b_i(z)-\lambda_i|g(z)\thinspace dz
    & \lesssim \prod_{i\in \bar{B}} \|b_i-\lambda_i\|_{\exp L,3Q} \hspace{1mm}\|g\|_{L(\log L)^{|\bar{B}|},3Q} \\
    & \lesssim \prod_{i\in \bar{B}} \|b_i\|_{BMO} \hspace{1mm} \|g\|_{L(\log L)^{|\bar{B}|},3Q}.
\end{split}
\end{equation*}
Then we have

\begin{align*}
   \lefteqn{\dashint_{3Q} \prod_{i\in A\cup B}|b_i(x)-\lambda_i|^\frac{q}{1-q}  u(x)^\frac{1}{1-q} \thinspace dx} \\
    &\quad \lesssim \prod_{i\in A\cup B}\bigl\||b_i-\lambda_i|^\frac{q}{1-q}\bigr\|_{\exp(L^\frac{1-q}{q}),3Q} \|u^\frac{1}{1-q}\|_{L(\log L)^\frac{q|A\cup B|}{1-q},3Q} \\
    & \quad\lesssim \prod_{i\in A\cup B} \|b_i\|_{BMO}^\frac{q}{1-q} \hspace{1mm} \|u^\frac{1}{1-q}\|_{L(\log L)^\frac{q|A\cup B|}{1-q},3Q}.
\end{align*}

Substituting these estimates into \eqref{eq43} and use the facts: $|\bar{A}|\leqslant m$, $|\bar{B}|\leqslant N-m$, $|A\cup B|\leqslant N$, and stronger Young functions provide bigger Orlicz norms, we come up with the following estimates.

\begin{align} 
\lefteqn{\int_{\mathbb{R}^n}\bigl|[\vec{b}, \mathsf{BI}_\alpha]_{\vec{\beta}}(f,g)(x)\bigr|^q u(x) \thinspace dx\lesssim \|\vec{b}\|^q \sum_{A\subseteq\{1,...,m\}}\sum_{B\subseteq\{m+1,...,N\}} \sum_{Q\in\mathscr{D}} |Q|^{\frac{\alpha}{n}q+1}} \nonumber \\
    &\hspace{3cm} \times\Bigl(\|f\|_{L(\log L)^m,3Q} \thinspace \|g\|_{L(\log L)^{N-m},3Q}\Bigr)^q \|u^\frac{1}{1-q}\|_{L(\log L)^\frac{qN}{1-q},3Q}^{1-q} \nonumber\\
    & \lesssim \|\vec{b}\|^q \sum_{A}\sum_{B} \sum_{t=1}^{3^n}\sum_{Q\in\mathscr{D}^t} |Q|^{\frac{\alpha}{n}q+1} \Bigl(\|f\|_{L(\log L)^m,Q} \thinspace \|g\|_{L(\log L)^{N-m},Q}\Bigr)^q\|u^\frac{1}{1-q}\|_{L(\log L)^\frac{qN}{1-q},Q}^{1-q}.\nonumber \\ \label{eq44}
\end{align}

The last inequality in \eqref{eq44} comes from the fact that each $3Q$ is contained in a $Q_t\in\mathscr{D}^t$, $t\in\{1,2,...,3^n\}$, with the property $\ell(3Q)\leqslant\ell(Q_t)\leqslant6\thinspace\ell(3Q)$. We note here that each $Q_t$ like that may contain more than one but at most $6^n$ such $3Q$'s where the $Q$'s are from a same layer of $\mathscr{D}$, and there are at most 3 possible layers. From here, it suffices to estimate inner most sum of the last expression in \eqref{eq44} for a generic dyadic grid $\mathscr{D}$. We will denote this sum as 
$$\mathsf{S}=\sum_{Q\in\mathscr{D}} |Q|^{\frac{\alpha}{n}q+1} \Bigl(\|f\|_{L(\log L)^m,Q} \thinspace \|g\|_{L(\log L)^{N-m},Q}\Bigr)^q  \|u^\frac{1}{1-q}\|_{L(\log L)^\frac{qN}{1-q},Q}^{1-q}.$$
 For simplicity, we will write $\mathscr{D}$ to mean any of the $\mathscr{D}^t$'s, $t=1,...,3^n$.

Next, we will replace the sum over dyadic cubes with the sum over a spare family of Calder\'on-Zygmund cubes. More precisely, for each $k\in\mathbb{Z}$, let $\{Q_j^k\}_j$ be a collection of disjoint dyadic cubes that are maximal with respect to

$$
\|f\|_{L(\log L)^m,Q_j^k} \thinspace \|g\|_{L(\log L)^{N-m},Q_j^k} > a^k,
$$
where $a>1$ will be chosen later.  This is possible because $\|f\|_{L(\log L)^m,Q}$ and $\|g\|_{L(\log L)^{N-m},Q}$ all tend to 0 as $\ell(Q)$ tends to $\infty$. Let $\Omega_k=\bigcupdot_jQ_j^k$ and $E_j^k=Q_j^k\setminus\Omega_{k+1}$, so that the family $\{E_j^k\}_{j,k}$ is pairwise disjoint and $\bigl|Q_j^k\bigr|\leqslant2\bigl|E_j^k\bigr|$. In fact,

\begin{align*}
    \lefteqn{\bigl|Q_j^k \cap \thinspace\Omega_{k+1}\bigr| = \sum_{Q_i^{k+1}\subseteq Q_j^k}\bigl|Q_i^{k+1}\bigr|} \\
    & \leqslant \frac{1}{a^{\frac{k+1}{2}}} \sum_{i} \Bigl(\bigl|Q_i^{k+1}\bigr| \thinspace \|f\|_{L(\log L)^m,Q_i^{k+1}}\Bigr)^\frac{1}{2} \Bigl(\bigl|Q_i^{k+1}\bigr| \thinspace \|g\|_{L(\log L)^{N-m},Q_i^{k+1}}\Bigr)^\frac{1}{2} \\
    & \leqslant \frac{1}{a^{\frac{k+1}{2}}} \left(\sum_{i} \bigl|Q_i^{k+1}\bigr| \thinspace \bigl\|f\bigr\|_{L(\log L)^m,Q_i^{k+1}}\right)^\frac{1}{2} \left(\sum_{i} \bigl|Q_i^{k+1}\bigr| \thinspace \|g\|_{L(\log L)^{N-m},Q_i^{k+1}}\right)^\frac{1}{2}. 
      \end{align*}
 If $\lambda,\mu>0$ then the previous sum is bounded by
  \begin{align*}  
    & \leqslant \frac{1}{a^{\frac{k+1}{2}}} \left[\sum_{i} \bigl|Q_i^{k+1}\bigr| \left(\lambda + \frac{\lambda}{\bigl|Q_i^{k+1}\bigr|}\int_{Q_i^{k+1}}\gamma_1\left(\frac{|f|}{\lambda}\right)\right)\right]^\frac{1}{2}\\
    & \hspace{2cm} \left[\sum_{i} \bigl|Q_i^{k+1}\bigr| \left(\mu + \frac{\mu}{\bigl|Q_i^{k+1}\bigr|}\int_{Q_i^{k+1}}\gamma_2\left(\frac{|g|}{\mu}\right)\right)\right]^\frac{1}{2}\\
   & = \frac{1}{a^{\frac{k+1}{2}}} \left[\sum_{i} \lambda \int_{Q_i^{k+1}}\left(1+\gamma_1\left(\frac{|f|}{\lambda}\right)\right)\right]^\frac{1}{2} \left[\sum_{i}\mu\int_{Q_i^{k+1}}\left(1+\gamma_2\left(\frac{|g|}{\mu}\right) \right)\right]^\frac{1}{2} \\
    & \leqslant \frac{1}{a^{\frac{k+1}{2}}} \left[\lambda \int_{Q_j^k}\left(1+\gamma_1\left(\frac{|f|}{\lambda}\right)\right)\right]^\frac{1}{2} \left[\mu\int_{Q_j^k}\left(1+\gamma_2\left(\frac{|g|}{\mu}\right) \right)\right]^\frac{1}{2} \\
    & \leqslant \frac{2^n}{a^{\frac{k+1}{2}}} \bigl|Q_j^k\bigr| \left[\lambda + \frac{\lambda}{\bigl|P\bigr|} \int_{P}\gamma_1\left(\frac{|f|}{\lambda}\right)\right]^\frac{1}{2} \left[\mu + \frac{\mu}{\bigl|P\bigr|} \int_{P}\gamma_2\left(\frac{|g|}{\mu}\right)\right]^\frac{1}{2}
\end{align*}
where $\gamma_1(t)=t\log(e+t)^m$, $\gamma_2(t)=t\log(e+t)^{N-m}$, and $P$ is an immediate dyadic parent of $Q_j^k$. By taking infimum over all $\lambda>0$ and all $\mu>0$, we have
\begin{equation*}
\begin{split}
    \bigl|Q_j^k\cap \Omega_{k+1}\bigr| & \leqslant \frac{2^{n+1}}{a^{\frac{k+1}{2}}} \bigl|Q_j^k\bigr| \Bigl(\|f\|_{L(\log L)^m,P}\thinspace \|g\|_{L(\log L)^{N-m},P}\Bigr)^\frac{1}{2} \\
    & \leqslant \frac{2^{n+1}}{a^{\frac{k+1}{2}}} \bigl|Q_j^k\bigr| \thinspace a^\frac{k}{2} \\
    & = \frac{2^{n+1}}{a^{\frac{1}{2}}} \bigl|Q_j^k\bigr| 
\end{split}
\end{equation*}
where the second inequality comes from the maximality of $Q_j^k$. With an appropriate choice of $a$, we will have $\bigl|Q_j^k\cap\thinspace\Omega_{k+1}\bigr|\leqslant\frac{1}{2}\bigl|Q_j^k\bigr|$, and hence $\bigl|Q_j^k\bigr|\leqslant2\bigl|E_j^k\bigr|$ as we wish. Now, let
$$
C_k=\left\{Q\in\mathscr{D}:\thinspace a^k < \|f\|_{L(\log L)^m,Q}\thinspace \|g\|_{L(\log L)^{N-m},Q} \leqslant a^{k+1}\right\}
$$
and notice that every $Q\in \mathscr{D}$ for which the summand of $\mathsf{S}$ is non-zero must be in some $C_k$, and every $Q\in C_k$ is contained in a unique $Q_j^k$. So we have

\begin{align} 
    \mathsf{S} & \leqslant \sum_{k\in\mathbb{Z}}\sum_{Q\in C_k} |Q|^{\frac{\alpha}{n}q+1} \Bigl(\|f\|_{L(\log L)^m,Q} \thinspace \|g\|_{L(\log L)^{N-m},Q}\Bigr)^q \|u^\frac{1}{1-q}\|_{L(\log L)^\frac{qN}{1-q},Q}^{1-q} \nonumber\\
    & \leqslant \sum_{k\in\mathbb{Z}} a^{(k+1)q} \sum_{Q\in C_k} |Q|^{\frac{\alpha}{n}q+1} \thinspace \bigl\|u^\frac{1}{1-q}\bigr\|_{L(\log L)^\frac{qN}{1-q},Q}^{1-q} \nonumber\\
    & \leqslant \sum_{k\in\mathbb{Z}} a^{(k+1)q} \sum_{j\in\mathbb{Z}} \sum_{\substack{Q\in \mathscr{D} \\ Q\subseteq Q_j^k}} |Q|^{\frac{\alpha}{n}q+1} \thinspace \bigl\|u^\frac{1}{1-q}\bigr\|_{L(\log L)^\frac{qN}{1-q},Q}^{1-q}.\label{eq45}
\end{align}
For each $\lambda>0$, the most inner sum is bounded by
\begin{align*}
\lefteqn{\leqslant \sum_{\substack{Q\in \mathscr{D} \\ Q\subseteq Q_j^k}} |Q|^{\frac{\alpha}{n}q+1} \left[ \lambda + \frac{\lambda}{|Q|}\int_{Q} \gamma\left(\frac{\bigl|u^\frac{1}{1-q}\bigr|}{\lambda}\right)\right]^{1-q},} \hspace{2mm}  \\
    & = \sum_{r=0}^{\infty}\sum_{\substack{Q\in \mathscr{D}, \thinspace Q\subseteq Q_j^k \\ \ell(Q)=2^{-r}\ell(Q_j^k)}} |Q|^{\left(\frac{\alpha}{n}+1\right)q} \left[ \lambda \int_{Q} \left(1+\gamma\left(\frac{\bigl|u^\frac{1}{1-q}\bigr|}{\lambda}\right)\right) \right]^{1-q} \\
    & = \bigl|Q_j^k\bigr|^{\left(\frac{\alpha}{n}+1\right)q} \sum_{r=0}^{\infty} 2^{-qr\alpha-qrn} \sum_{\substack{Q\in \mathscr{D}, \thinspace Q\subseteq Q_j^k \\ \ell(Q)=2^{-r}\ell(Q_j^k)}} \left[\lambda \int_{Q} \left(1+\gamma\left(\frac{\bigl|u^\frac{1}{1-q}\bigr|}{\lambda}\right)\right) \right]^{1-q} \\
& \qquad \times \bigl|Q_j^k\bigr|^{\left(\frac{\alpha}{n}+1\right)q} \sum_{r=0}^{\infty} 2^{-qr\alpha-qrn} \left[\sum_{\substack{Q\in \mathscr{D}, \thinspace Q\subseteq Q_j^k \\ \ell(Q)=2^{-r}\ell(Q_j^k)}} \lambda \int_{Q} \left(1+\gamma\left(\frac{\bigl|u^\frac{1}{1-q}\bigr|}{\lambda}\right)\right) \right]^{1-q} \\
    & \hspace{9cm} \left(\sum_{\substack{Q\in \mathscr{D}, \thinspace Q\subseteq Q_j^k \\ \ell(Q)=2^{-r}\ell(Q_j^k)}} 1 \right)^q \\
    & = \bigl|Q_j^k\bigr|^{\left(\frac{\alpha}{n}+1\right)q} \left[\lambda \int_{Q_j^k} \left(1+\gamma\left(\frac{\bigl|u^\frac{1}{1-q}\bigr|}{\lambda}\right)\right) \right]^{1-q} \sum_{r=0}^{\infty} 2^{-qr\alpha} \\
    & = \frac{2^{\alpha q}}{2^{\alpha q}-1}\thinspace \bigl|Q_j^k\bigr|^{\frac{\alpha}{n}q+1} \left[\lambda + \frac{\lambda}{\bigl|Q_j^k\bigr|} \int_{Q_j^k}\gamma\left(\frac{\bigl|u^\frac{1}{1-q}\bigr|}{\lambda}\right)\right]^{1-q}, 
\end{align*}
where $\gamma(t)=t\log(e+t)^\frac{qN}{1-q}$. By taking infimum over all $\lambda>0$ and then substituting the result into \eqref{eq45}, we have

\begin{equation} \label{eq46}
\begin{split}
    \mathsf{S} & \lesssim \sum_{k\in\mathbb{Z}} a^{(k+1)q} \sum_{j\in\mathbb{Z}} \thinspace \bigl|Q_j^k\bigr|^{\frac{\alpha}{n}q+1} \thinspace \bigl\|u^\frac{1}{1-q}\bigr\|_{L(\log L)^\frac{qN}{1-q},Q_j^k}^{1-q} \\
    & \lesssim \sum_{k,j} \thinspace \bigl|Q_j^k\bigr|^{\frac{\alpha}{n}q+1} \Bigl(\|f\|_{L(\log L)^m,Q_j^k} \thinspace \|g\|_{L(\log L)^{N-m},Q_j^k}\Bigr)^q \|u^\frac{1}{1-q}\|_{L(\log L)^\frac{qN}{1-q},Q_j^k}^{1-q}.
\end{split}
\end{equation}

Now we consider the following Young functions.
\begin{equation*}
    \tau_1(t)  =\frac{t^{p_1}}{\log(e+t)^{1+(p_1-1)\delta}} \ \ \text{and} \ \   \tau_2(t) =\frac{t^{p_2}}{\log(e+t)^{1+(p_2-1)\delta}}.
\end{equation*}
Straightforward calculations show that $\tau_1\in B_{p_1}$, $\tau_2\in B_{p_2}$, and
\begin{equation*}
    \tau_1^{-1}(t) \thinspace \phi_1^{-1}(t) \approx \frac{t}{\log(e+t)^m} \ \ \text{and} \ \     \tau_2^{-1}(t) \thinspace \phi_2^{-1}(t) \approx \frac{t}{\log(e+t)^{N-m}}.
\end{equation*}

Using the generalized H\"older inequality and the imposed conditions on the weights, from \eqref{eq46} we have

\begin{equation*}
\begin{split}
    \mathsf{S} & \lesssim \sum_{k,j} \thinspace \bigl|Q_j^k\bigr|^{\frac{\alpha}{n}q+1} \Bigl(\bigl\|fv_1^\frac{1}{p_1}\bigr\|_{\tau_1,Q_j^k} \thinspace \bigl\|v_1^{-\frac{1}{p_1}}\bigr\|_{\phi_1,Q_j^k} \thinspace \bigl\|gv_2^\frac{1}{p_2}\bigr\|_{\tau_2,Q_j^k} \thinspace \bigl\|v_2^{-\frac{1}{p_2}}\bigr\|_{\phi_2,Q_j^k}\Bigr)^q \\
    & \hspace{9.2cm} \bigl\|u^\frac{1}{1-q}\bigr\|_{L(\log L)^\frac{qN}{1-q},Q_j^k}^{1-q} \\
    & \lesssim \sum_{k,j} \thinspace \bigl|Q_j^k\bigr|^{\frac{q}{p}} \Bigl(\bigl\|fv_1^\frac{1}{p_1}\bigr\|_{\tau_1,Q_j^k} \thinspace \bigl\|gv_2^\frac{1}{p_2}\bigr\|_{\tau_2,Q_j^k}\Bigr)^q. \\
\end{split}
\end{equation*}

From here, we are going to use: the condition that $p\leqslant q$, the fact that $\bigl|Q_j^k\bigr|\leqslant2\bigl|E_j^k\bigr|$, discrete H\"older inequality with the pair  $\left(\frac{p_1}{p},\frac{p_2}{p}\right)$, and theorem \ref{thm3} to obtain the following estimates.
\begin{equation} \label{eq47}
\begin{split}
    \mathsf{S} & \lesssim \left(\sum_{k,j} \thinspace \bigl|Q_j^k\bigr| \thinspace \bigl\|fv_1^\frac{1}{p_1}\bigr\|^p_{\tau_1,Q_j^k} \thinspace \bigl\|gv_2^\frac{1}{p_2}\bigr\|^p_{\tau_2,Q_j^k}\right)^{\frac{q}{p}} \\
    & \lesssim \left(\sum_{k,j} \thinspace \bigl|E_j^k\bigr|^\frac{p}{p_1} \thinspace \bigl\|fv_1^\frac{1}{p_1}\bigr\|^p_{\tau_1,Q_j^k} \thinspace \bigl|E_j^k\bigr|^\frac{p}{p_2} \thinspace \bigl\|gv_2^\frac{1}{p_2}\bigr\|^p_{\tau_2,Q_j^k}\right)^{\frac{q}{p}} \\
    & \leqslant \left(\sum_{k,j} \thinspace \bigl|E_j^k\bigr| \thinspace \bigl\|fv_1^\frac{1}{p_1}\bigr\|^{p_1}_{\tau_1,Q_j^k}\right)^{\frac{q}{p_1}} \left(\sum_{k,j} \thinspace \bigl|E_j^k\bigr| \thinspace \bigl\|gv_2^\frac{1}{p_2}\bigr\|^{p_2}_{\tau_2,Q_j^k}\right)^{\frac{q}{p_2}} \\
    & \leqslant \left(\sum_{k,j} \int_{E_j^k} M_{\tau_1}\Bigl(fv_1^\frac{1}{p_1}\Bigr)(x)^{p_1}dx\right)^{\frac{q}{p_1}} \left(\sum_{k,j} \int_{E_j^k} M_{\tau_2}\Bigl(gv_2^\frac{1}{p_2}\Bigr)(x)^{p_2}dx\right)^{\frac{q}{p_2}} \\
    & \leqslant \left(\int_{\mathbb{R}^n} M_{\tau_1}\Bigl(fv_1^\frac{1}{p_1}\Bigr)(x)^{p_1}dx\right)^{\frac{q}{p_1}} \left(\int_{\mathbb{R}^n} M_{\tau_2}\Bigl(gv_2^\frac{1}{p_2}\Bigr)(x)^{p_2}dx\right)^{\frac{q}{p_2}} \\
    & \lesssim \|f\|^q_{L^{p_1}(v_1)} \|g\|^q_{L^{p_2}(v_2)}.
\end{split}
\end{equation}

\hfill

Substituting the result in \eqref{eq47} into \eqref{eq44} will give us the desired estimate

$$
\bigl\|[\vec{b},\mathsf{BI}_{\alpha}]_{\vec{\beta}}(f,g)\bigr\|_{L^q(u)} \lesssim \|\vec{b}\|\left\|f\right\|_{L^{p_1}(v_1)} \left\|g\right\|_{L^{p_2}(v_2)}.
$$

%-------------------------------------------------
% NEW SECTION: PROOF OF THEOREM 2.4
%-------------------------------------------------
\section{Proof of Theorem \ref{thmB}}

By duality, it suffices to prove that for all $f\in L^{p_1}(v_1)$, all $g\in L^{p_2}(v_2)$ and all  $h\in L^{q'}(\mathbb{R}^n)$ with $\|h\|_{q'}=1$,

\begin{equation*}
\int_{\mathbb{R}^n}\bigl|[\vec{b}, \mathsf{BI}_\alpha]_{\vec{\beta}}(f,g)(x)\bigr| \thinspace h(x) \thinspace u(x)^\frac{1}{q} \thinspace dx \lesssim \|\vec{b}\|\left\|f\right\|_{L^{p_1}(v_1)} \left\|g\right\|_{L^{p_2}(v_2)}.
\end{equation*}

\hfill

Without loss of generality, we may assume that $f$ and $g$ are non-negative, bounded and compactly supported. From \eqref{eq42}, we have

\begin{align*}
\lefteqn{ \int_{\mathbb{R}^n}\bigl|[\vec{b}, \mathsf{BI}_\alpha]_{\vec{\beta}}(f,g)(x)\bigr| \thinspace h(x) \thinspace u(x)^\frac{1}{q} \thinspace dx} \\
    & \lesssim \sum_{A\subseteq\{1,...,m\}}\sum_{B\subseteq\{m+1,...,N\}} \sum_{Q\in\mathscr{D}} |Q|^{\frac{\alpha}{n}-1} \\
    & \hspace{1.7cm} \int_Q\int_{|y|_\infty\leqslant\ell(Q)} \prod_{i\in \bar{A}}|b_i(x-y)-\lambda_i| \prod_{i\in \bar{B}}|b_i(x+y)-\lambda_i| f(x-y)g(x+y)\thinspace dy \\
    & \hspace{7.6cm} \prod_{i\in A\cup B}|b_i(x)-\lambda_i| \thinspace \thinspace h(x) \thinspace u(x)^\frac{1}{q} \thinspace dx.
\end{align*}

If we use H\"older inequality with the pair $(r,s)$ for the inner integral, and then perform a change in variables, we will get

\begin{align} 
\lefteqn{\int_{\mathbb{R}^n}\bigl|[\vec{b}, \mathsf{BI}_\alpha]_{\vec{\beta}}(f,g)(x)\bigr| \thinspace h(x) \thinspace u(x)^\frac{1}{q} \thinspace dx}\nonumber \\
    & \lesssim \sum_{A\subseteq\{1,...,m\}}\sum_{B\subseteq\{m+1,...,N\}} \sum_{Q\in\mathscr{D}} |Q|^{\frac{\alpha}{n}-1} \nonumber\\
    & \hspace{2cm} \left[\int_{3Q} \prod_{i\in \bar{A}}|b_i(t)-\lambda_i|^rf(t)^r dt\right]^\frac{1}{r} \left[\int_{3Q} \prod_{i\in \bar{B}}|b_i(z)-\lambda_i|^sg(z)^s dz\right]^\frac{1}{s} \nonumber\\
    & \hspace{7.1cm} \int_Q\prod_{i\in A\cup B}|b_i(x)-\lambda_i| \thinspace \thinspace h(x) \thinspace u(x)^\frac{1}{q} \thinspace dx.\nonumber \\ \label{eq51}
\end{align}
We now use the generalized H\"older inequality, Theorem \ref{thm7} and Corollary \ref{thm8} to get the following estimates.

\begin{equation*}
\begin{split}
    \dashint_{3Q} \prod_{i\in \bar{A}}|b_i(t)-\lambda_i|^rf(t)^r\thinspace dt
    & \lesssim \prod_{i\in \bar{A}} \||b_i-\lambda_i|^r\|_{\exp(L^\frac{1}{r}),3Q} \hspace{1mm}\|f^r\|_{L(\log L)^{r|\bar{A}|},3Q} \\
    & \lesssim \prod_{i\in \bar{A}} \|b_i\|^r_{BMO} \hspace{1mm} \|f^r\|_{L(\log L)^{r|\bar{A}|},3Q}
\end{split}
\end{equation*}
\begin{equation*}
\begin{split}
    \dashint_{3Q} \prod_{i\in \bar{B}}|b_i(z)-\lambda_i|^sg(z)^s\thinspace dz
    & \lesssim \prod_{i\in \bar{B}} \||b_i-\lambda_i|^s\|_{\exp(L^\frac{1}{s}),3Q} \hspace{1mm}\|g^s\|_{L(\log L)^{s|\bar{B}|},3Q} \\
    & \lesssim \prod_{i\in \bar{B}} \|b_i\|^s_{BMO} \hspace{1mm} \|g^s\|_{L(\log L)^{s|\bar{B}|},3Q}
\end{split}
\end{equation*}
\begin{equation*}
\begin{split}
    \dashint_Q\prod_{i\in A\cup B}|b_i(x)-\lambda_i| \thinspace \thinspace h(x) \thinspace u(x)^\frac{1}{q} \thinspace dx & \lesssim \prod_{i\in A\cup B}\|b_i-\lambda_i\|_{\exp L,3Q} \thinspace \|hu^\frac{1}{q}\|_{L(\log L)^{|A\cup B|},3Q} \\
    & \lesssim \prod_{i\in A\cup B} \|b_i\|_{BMO} \thinspace \|hu^\frac{1}{q}\|_{L(\log L)^{|A\cup B|},3Q}.
\end{split}
\end{equation*}

Substituting these estimates into \eqref{eq51} and using the facts: $|\bar{A}|\leqslant m$, $|\bar{B}|\leqslant N-m$, $|A\cup B|\leqslant N$, and stronger Young functions provide bigger Orlicz norms, we come up with the following estimates.

\begin{align} 
\lefteqn{ \int_{\mathbb{R}^n}\bigl|[\vec{b}, \mathsf{BI}_\alpha]_{\vec{\beta}}(f,g)(x)\bigr| \thinspace h(x) \thinspace u(x)^\frac{1}{q} \thinspace dx} \nonumber\\
    & \lesssim \|\vec{b}\| \sum_{A\subseteq\{1,...,m\}}\sum_{B\subseteq\{m+1,...,N\}} \sum_{Q\in\mathscr{D}} |Q|^{\frac{\alpha}{n}+1} \|f^r\|^\frac{1}{r}_{L(\log L)^{mr},3Q} \thinspace \|g^s\|^\frac{1}{s}_{L(\log L)^{(N-m)s},3Q}\nonumber\\
    & \hspace{2.5cm} \times \|hu^\frac{1}{q}\|_{L(\log L)^N,3Q}\nonumber \\
    & \lesssim \|\vec{b}\| \sum_{A}\sum_{B} \sum_{t=1}^{3^n} \sum_{Q\in\mathscr{D}} |Q|^{\frac{\alpha}{n}+1} \|f^r\|^\frac{1}{r}_{L(\log L)^{mr},Q} \thinspace \|g^s\|^\frac{1}{s}_{L(\log L)^{(N-m)s},Q} \thinspace \|hu^\frac{1}{q}\|_{L(\log L)^N,Q}.\nonumber \\ \label{eq52}
\end{align}

It suffices to control the inner most sum of the last expression in \eqref{eq52} for a general dyadic grid $\mathscr{D}$.  To do so, we will replace the sum over dyadic cubes with the sum over a spare family of Calder\'on-Zygmund cubes. Let $a>1$ be a number that will be chosen later. For each $k\in\mathbb{Z}$, let $\{Q_j^k\}_j$ be a collection of disjoint dyadic cubes that are maximal with respect to
$$
\|f^r\|^\frac{1}{r}_{L(\log L)^{mr},Q} \thinspace \|g^s\|^\frac{1}{s}_{L(\log L)^{(N-m)s},Q} > a^k.
$$

Let $\Omega_k=\bigcupdot_jQ_j^k$ and $E_j^k=Q_j^k\setminus\Omega_{k+1}$, so that the family $\{E_j^k\}_{j,k}$ is pairwise disjoint and $\bigl|Q_j^k\bigr|\leqslant2\bigl|E_j^k\bigr|$. In fact,

\begin{align*}
\lefteqn{\bigl|Q_j^k \cap \thinspace\Omega_{k+1}\bigr| = \sum_{Q_i^{k+1}\subseteq Q_j^k}\bigl|Q_i^{k+1}\bigr|} \\
    & \leqslant \frac{1}{a^{k+1}} \sum_{i} \Bigl(\bigl|Q_i^{k+1}\bigr| \thinspace \|f^r\|_{L(\log L)^{mr},Q_i^{k+1}}\Bigr)^\frac{1}{r} \Bigl(\bigl|Q_i^{k+1}\bigr| \thinspace \|g^s\|_{L(\log L)^{(N-m)s},Q_i^{k+1}}\Bigr)^\frac{1}{s}. \\
    & \leqslant \frac{1}{a^{k+1}} \left(\sum_{i} \bigl|Q_i^{k+1}\bigr| \thinspace \|f^r\|_{L(\log L)^{mr},Q_i^{k+1}}\right)^\frac{1}{r} \left(\sum_{i} \bigl|Q_i^{k+1}\bigr| \thinspace \|g^s\|_{L(\log L)^{(N-m)s},Q_i^{k+1}}\right)^\frac{1}{s}.
     \end{align*}
 If $\lambda,\mu>0$ then the previous sum is bounded by
 \begin{align*}
    & \leqslant \frac{1}{a^{k+1}} \left[\sum_{i} \bigl|Q_i^{k+1}\bigr| \left(\lambda + \frac{\lambda}{\bigl|Q_i^{k+1}\bigr|}\int_{Q_i^{k+1}}\gamma_1\left(\frac{|f|^r}{\lambda}\right)\right)\right]^\frac{1}{r} \\
    & \hspace{2.5cm} \left[\sum_{i} \bigl|Q_i^{k+1}\bigr| \left(\mu + \frac{\mu}{\bigl|Q_i^{k+1}\bigr|}\int_{Q_i^{k+1}}\gamma_2\left(\frac{|g|^s}{\mu}\right)\right)\right]^\frac{1}{s}, \hspace{1mm} \forall \mu>0 \\
     & = \frac{1}{a^{k+1}} \left[\sum_{i} \lambda \int_{Q_i^{k+1}}\left(1+\gamma_1\left(\frac{|f|^r}{\lambda}\right)\right)\right]^\frac{1}{r} \left[\sum_{i}\mu\int_{Q_i^{k+1}}\left(1+\gamma_2\left(\frac{|g|^s}{\mu}\right) \right)\right]^\frac{1}{s}
\end{align*}

\begin{equation*}
\begin{split}
    & \leqslant \frac{1}{a^{k+1}} \left[\lambda \int_{Q_j^k}\left(1+\gamma_1\left(\frac{|f^r|}{\lambda}\right)\right)\right]^\frac{1}{r} \left[\mu\int_{Q_j^k}\left(1+\gamma_2\left(\frac{|g^s|}{\mu}\right) \right)\right]^\frac{1}{s} \\
    & \leqslant \frac{2^n}{a^{k+1}} \bigl|Q_j^k\bigr| \left[\lambda + \frac{\lambda}{\bigl|P\bigr|} \int_{P}\gamma_1\left(\frac{|f^r|}{\lambda}\right)\right]^\frac{1}{r} \left[\mu + \frac{\mu}{\bigl|P\bigr|} \int_{P}\gamma_2\left(\frac{|g^s|}{\mu}\right)\right]^\frac{1}{s}
\end{split}
\end{equation*}
where $\gamma_1(t)=t\log(e+t)^{mr}$, $\gamma_2(t)=t\log(e+t)^{(N-m)s}$, and $P$ is an immediate dyadic parent of $Q_j^k$. By taking infimum over all $\lambda>0$ and all $\mu>0$, we have
\begin{equation*}
\begin{split}
    \bigl|Q_j^k\cap \Omega_{k+1}\bigr| & \leqslant \frac{2^{n+1}}{a^{k+1}} \bigl|Q_j^k\bigr| \|f\|^\frac{1}{r}_{L(\log L)^{mr},P}\thinspace \|g\|^\frac{1}{s}_{L(\log L)^{(N-m)s},P} \\
    & \leqslant \frac{2^{n+1}}{a^{k+1}} \bigl|Q_j^k\bigr| \thinspace a^k = \frac{2^{n+1}}{a}\bigl|Q_j^k\bigr| 
\end{split}
\end{equation*}
where the last inequality comes from the maximality of $Q_j^k$. With an appropriate choice of $a$, we will have $\bigl|Q_j^k\bigr|\leqslant2\bigl|E_j^k\bigr|$. Now, let
$$
C_k=\left\{Q\in\mathscr{D}:\thinspace a^k < \|f^r\|^\frac{1}{r}_{L(\log L)^{mr},Q} \thinspace \|g^s\|^\frac{1}{s}_{L(\log L)^{(N-m)s},Q} \leqslant a^{k+1}\right\}
$$
and notice that every $Q\in \mathscr{D}$ for which the summand of $\mathsf{S}$ is non-zero must be in some $C_k$, and every $Q\in C_k$ is contained in a unique $Q_j^k$. So we have

\begin{equation} \label{eq53}
\begin{split}
    \mathsf{S}& \leqslant \sum_{k\in\mathbb{Z}}\sum_{Q\in C_k} |Q|^{\frac{\alpha}{n}+1} \|f^r\|^\frac{1}{r}_{L(\log L)^{mr},Q} \thinspace \|g^s\|^\frac{1}{s}_{L(\log L)^{(N-m)s},Q} \|hu^\frac{1}{q}\|_{L(\log L)^N,Q} \\
    & \leqslant \sum_{k\in\mathbb{Z}} a^{k+1} \sum_{Q\in C_k} |Q|^{\frac{\alpha}{n}+1} \thinspace \|hu^\frac{1}{q}\|_{L(\log L)^N,Q} \\
    & \leqslant \sum_{k\in\mathbb{Z}} a^{k+1} \sum_{j\in\mathbb{Z}} \sum_{\substack{Q\in \mathscr{D} \\ Q\subseteq Q_j^k}} |Q|^{\frac{\alpha}{n}+1} \thinspace \|hu^\frac{1}{q}\|_{L(\log L)^N,Q}.
\end{split}
\end{equation}
For all $\lambda>0$ the most inner sum is bounded by
\begin{equation*}
\begin{split}
    & \leqslant \sum_{\substack{Q\in \mathscr{D} \\ Q\subseteq Q_j^k}} |Q|^{\frac{\alpha}{n}+1} \left[ \lambda + \frac{\lambda}{|Q|}\int_{Q} \gamma\left(\frac{\bigl|hu^\frac{1}{q}\bigr|}{\lambda}\right)\right], \hspace{2mm}  \\
    & = \sum_{r=0}^{\infty}\sum_{\substack{Q\in \mathscr{D}, \thinspace Q\subseteq Q_j^k \\ \ell(Q)=2^{-r}\ell(Q_j^k)}} |Q|^{\frac{\alpha}{n}}  \lambda \int_{Q} \left[1+\gamma\left(\frac{\bigl|hu^\frac{1}{q}\bigr|}{\lambda}\right)\right] \\
    & = \lambda \bigl|Q_j^k\bigr|^{\frac{\alpha}{n}} \sum_{r=0}^{\infty} 2^{-r\alpha} \sum_{\substack{Q\in \mathscr{D}, \thinspace Q\subseteq Q_j^k \\ \ell(Q)=2^{-r}\ell(Q_j^k)}} \int_{Q} \left[1+\gamma\left(\frac{\bigl|hu^\frac{1}{q}\bigr|}{\lambda}\right)\right] \\
    & = \frac{2^{\alpha}}{2^{\alpha}-1}\thinspace \bigl|Q_j^k\bigr|^{\frac{\alpha}{n}+1} \left[\lambda + \frac{\lambda}{\bigl|Q_j^k\bigr|} \int_{Q_j^k}\gamma\left(\frac{\bigl|hu^\frac{1}{q}\bigr|}{\lambda}\right)\right], \hspace{2mm} 
\end{split}
\end{equation*}
where $\gamma(t)=t\log(e+t)^N$. By taking infimum over all $\lambda>0$ and then substituting the result into \eqref{eq53}, we end up having

\begin{equation} \label{eq54}
\begin{split}
   \mathsf{S} & \lesssim \sum_{k\in\mathbb{Z}} a^{k+1} \sum_{j\in\mathbb{Z}} \thinspace \bigl|Q_j^k\bigr|^{\frac{\alpha}{n}+1} \thinspace \|hu^\frac{1}{q}\|_{L(\log L)^N,Q_j^k} \\
    & \lesssim \sum_{k,j} \thinspace \bigl|Q_j^k\bigr|^{\frac{\alpha}{n}+1} \|f^r\|^\frac{1}{r}_{L(\log L)^{mr},Q_j^k} \thinspace \|g^s\|^\frac{1}{s}_{L(\log L)^{(N-m)s},Q_j^k} \thinspace \|hu^\frac{1}{q}\|_{L(\log L)^N,Q_j^k}.
\end{split}
\end{equation}

Now we consider the following Young functions.
\begin{equation*}
\begin{split}
    \tau_1(t) & =\frac{t^{\frac{p_1}{r}}}{\log(e+t)^{1+(\frac{p_1}{r}-1)\delta}} \\
    \tau_2(t) & =\frac{t^{\frac{p_2}{s}}}{\log(e+t)^{1+(\frac{p_2}{s}-1)\delta}} \\
    \tau(t) & =\frac{t^{q'}}{\log(e+t)^{1+(q'-1)\delta}}
\end{split}
\end{equation*}

Straight forward calculations show that $\tau_1\in B_{\frac{p_1}{r}}$, $\tau_2\in B_{\frac{p_2}{s}}$, $\tau\in B_{q'}$, and
\begin{equation*}
\begin{split}
    \tau_1^{-1}(t) \thinspace \phi_1^{-1}(t) &\approx \frac{t}{\log(e+t)^{mr}}  \\
    \tau_2^{-1}(t) \thinspace \phi_2^{-1}(t) &\approx \frac{t}{\log(e+t)^{(N-m)s}}   \\
    \tau^{-1}(t) \thinspace \psi^{-1}(t) &\approx \frac{t}{\log(e+t)^N}. 
\end{split}
\end{equation*}

So, by using the generalized H\"older inequality and the imposed conditions on the weights, from \eqref{eq54} we have

\begin{equation*}
\begin{split}
\mathsf{S}& \lesssim \sum_{k,j} \thinspace \bigl|Q_j^k\bigr|^{\frac{\alpha}{n}+1} \bigl\|f^rv_1^\frac{r}{p_1}\bigr\|^\frac{1}{r}_{\tau_1,Q_j^k} \thinspace \bigl\|v_1^{-\frac{r}{p_1}}\bigr\|^\frac{1}{r}_{\phi_1,Q_j^k} \thinspace \bigl\|g^sv_2^\frac{s}{p_2}\bigr\|^\frac{1}{s}_{\tau_2,Q_j^k} \thinspace \bigl\|v_2^{-\frac{s}{p_2}}\bigr\|^\frac{1}{s}_{\phi_2,Q_j^k} \\
    & \hspace{9.2cm} \|h\|_{\tau,Q_j^k} \thinspace \|u^\frac{1}{q}\|_{\psi,Q_j^k} \\
    & \lesssim \sum_{k,j} \thinspace \bigl|Q_j^k\bigr|^{\frac{1}{p}+\frac{1}{q'}} \bigl\|f^rv_1^\frac{r}{p_1}\bigr\|^\frac{1}{r}_{\tau_1,Q_j^k} \thinspace \bigl\|g^sv_2^\frac{s}{p_2}\bigr\|^\frac{1}{s}_{\tau_2,Q_j^k} \thinspace \|h\|_{\tau,Q_j^k}.
\end{split}
\end{equation*}
We are going to use: the fact that $\bigl|Q_j^k\bigr|\leqslant2\bigl|E_j^k\bigr|$, Proposition \ref{thm9} with the triple $(p_1,p_2,q')$, and Theorem \ref{thm3} to obtain the following estimates.

\begin{align} 
   \lefteqn{\mathsf{S} \lesssim \sum_{k,j} \thinspace  \left(\bigl\|f^rv_1^\frac{r}{p_1}\bigr\|^\frac{1}{r}_{\tau_1,Q_j^k} \thinspace \bigl|E_j^k\bigr|^{\frac{1}{p_1}}\right) \left(\bigl\|g^sv_2^\frac{s}{p_2}\bigr\|^\frac{1}{s}_{\tau_2,Q_j^k} \thinspace \bigl|E_j^k\bigr|^{\frac{1}{p_2}}\right) \left(\|h\|_{\tau,Q_j^k} \thinspace \bigl|E_j^k\bigr|^{\frac{1}{q'}}\right)}\nonumber \\
    & \leqslant \left[\sum_{k,j} \thinspace \bigl\|f^rv_1^\frac{r}{p_1}\bigr\|^\frac{p_1}{r}_{\tau_1,Q_j^k} \thinspace \bigl|E_j^k\bigr|\right]^\frac{1}{p_1} \left[\sum_{k,j} \thinspace \bigl\|g^sv_2^\frac{s}{p_2}\bigr\|^\frac{p_2}{s}_{\tau_2,Q_j^k} \thinspace \bigl|E_j^k\bigr|\right]^\frac{1}{p_2} \left[\sum_{k,j} \thinspace \|h\|^{q'}_{\tau,Q_j^k} \thinspace \bigl|E_j^k\bigr|\right]^\frac{1}{q'}  \nonumber\\ \label{eq55}
\end{align}

\begin{equation*}
\begin{split}
    & \leqslant \left[\sum_{k,j} \int_{E_j^k} M_{\tau_1}\Bigl(f^rv_1^\frac{r}{p_1}\Bigr)(x)^\frac{p_1}{r}dx\right]^{\frac{1}{p_1}} \left[\sum_{k,j} \int_{E_j^k} M_{\tau_2}\Bigl(g^sv_2^\frac{s}{p_2}\Bigr)(x)^\frac{p_2}{s}dx\right]^{\frac{1}{p_2}} \\
    & \hspace{8.2cm} \left[\sum_{k,j} \int_{E_j^k} M_{\tau}(h)(x)^{q'}dx\right]^{\frac{1}{q'}}\\
    & \leqslant \left[\int_{\mathbb{R}^n} M_{\tau_1}\Bigl(f^rv_1^\frac{r}{p_1}\Bigr)(x)^\frac{p_1}{r}dx\right]^{\frac{1}{p_1}} \left[\int_{\mathbb{R}^n} M_{\tau_2}\Bigl(g^sv_2^\frac{s}{p_2}\Bigr)(x)^\frac{p_2}{s}dx\right]^{\frac{1}{p_2}} \\
    & \hspace{8.9cm} \left[\int_{\mathbb{R}^n} M_{\tau}(h)(x)^{q'}dx\right]^{\frac{1}{q'}} \\
    & \lesssim \|f\|_{L^{p_1}(v_1)} \|g\|_{L^{p_2}(v_2)} \|h\|_{q'}.
\end{split}
\end{equation*}
Since $\mathsf{S}$ can be any term in \eqref{eq52}, and the number of terms in \eqref{eq52} is finite, substituting the result in \eqref{eq55} into \eqref{eq52} will complete our proof.

\hfill

%-------------------------------------------------
% NEW SECTION: PROOF OF THEOREM 2.6
%-------------------------------------------------
\section{Proof of Theorem \ref{thmC}}

Again, without loss of generality, we may restrict ourselves  onto working with $f$ and $g$ that are non-negative, bounded and compactly supported. Thanks to Theorem \ref{thm6}, we only need to verify the inequality for a certain $q_0\in(0,\infty)$ and an arbitrary weight $w\in A_\infty$. We will work with $q_0=1$. By mimicking what we did in the proof of Theorem \ref{thmB}, we have

\begin{multline} 
   \int_{\mathbb{R}^n}\bigl|[\vec{b},  \mathsf{BT}_\alpha]_{\vec{\beta}}(f,g)(x)\bigr| \thinspace w(x) \thinspace dx \\
     \lesssim \|\vec{b}\| \sum_{k,j} |Q_j^k|^{\frac{\alpha}{n}+1} \|f^r\|^\frac{1}{r}_{L(\log L)^{mr},Q_j^k} \thinspace \|g^s\|^\frac{1}{s}_{L(\log L)^{(N-m)s},Q_j^k} \thinspace \|w\|_{L(\log L)^N,Q_j^k}\label{eq61}.
\end{multline}

Since $w\in A_\infty$, there exist, by Lemma \ref{thm10}, a number $m>1$ such that
$$
\left(\dashint_Qw^{m}\right)^{\frac{1}{m}} \lesssim \thinspace \dashint_Qw.
$$
The Young function $\psi(t)=t^m$ is stronger than $\phi(t)=t\log(e+t)^N$, which implies
$$
\|w\|_{L(\log L)^N,Q_j^k} \lesssim \left(\dashint_Qw^{m}\right)^{\frac{1}{m}} \lesssim \thinspace \dashint_Qw.
$$
Substituting this result into \eqref{eq61}, we have

\begin{equation*}
\begin{split}
   \int_{\mathbb{R}^n}\bigl|[\vec{b},\mathsf{BT}_\alpha &]_{\vec{\beta}}(f,g)(x)\bigr| \thinspace w(x) \thinspace dx \\
    & \lesssim \|\vec{b}\| \sum_{k,j} |Q_j^k|^{\frac{\alpha}{n}} \|f^r\|^\frac{1}{r}_{L(\log L)^{mr},Q_j^k} \thinspace \|g^s\|^\frac{1}{s}_{L(\log L)^{(N-m)s},Q_j^k} \thinspace w(Q_j^k) \\
    & \lesssim \|\vec{b}\| \sum_{k,j} |Q_j^k|^{\frac{\alpha}{n}} \|f^r\|^\frac{1}{r}_{L(\log L)^{mr},Q_j^k} \thinspace \|g^s\|^\frac{1}{s}_{L(\log L)^{(N-m)s},Q_j^k} \thinspace w(E_j^k) \\
    & \leqslant \|\vec{b}\| \sum_{k,j} \int_{E_j^k} \mathcal{M}^{r,s}_{\alpha}(f,g)(x) \thinspace w(x) \thinspace dx \\
    & \leqslant \|\vec{b}\| \int_{\mathbb{R}^n} \mathcal{M}^{r,s}_{\alpha}(f,g)(x) \thinspace w(x) \thinspace dx
\end{split}
\end{equation*}
where the second inequality is due to Lemma \ref{thm10} and the fact that $w\in A_\infty$.

\hfill

%-------------------------------------------------
% NEW SECTION: PROOF OF THEOREM 2.7
%-------------------------------------------------
\section{Proof of Theorem \ref{thmG}}

\noindent [\textit{Condition \eqref{eq21} $\Rightarrow$ the weak type boundedness}] 

In light of Theorem \ref{thm1}, it is not hard to see
\begin{equation*}
  \M^{r,s}_{\alpha}(f,g)(x) \leqslant 6^{n-\alpha} \sum_{t\in \{0,1/3\}^n}{\M^{r,s,\mathscr{D}^t}_{\alpha}(f,g)(x)}
\end{equation*}
where
$$
\M^{r,s,\mathscr{D}}_{\alpha}(f,g)(x)=\sup_{\mathscr{D}\ni Q\ni x}{|Q|^{\frac{\alpha}{n}}\left(\dashint_Q{|f|^r}\right)^{\frac{1}{r}}\left(\dashint_Q{|g|^s}\right)^{\frac{1}{s}}}.
$$

So, we will only need to prove the weak type boundedness for $\M^{r,s,\mathscr{D}}_{\alpha}$ where $\mathscr{D}$ is an arbitrary dyadic grid. Without loss of generality, we may assume that $f,g$ are non-negative, bounded and compactly supported. By performing the Calder\'on-Zygmund decomposition algorithm, we have
\begin{equation} \label{eq71}
    E_{\lambda}=\{x\in \mathbb{R}^n:\thinspace \M^{r,s,\mathscr{D}}_{\alpha}(f,g)(x)>\lambda\} = \bigcupdot_j{Q_j}
\end{equation}
where $Q_j$'s are pairwise disjoint maximal dyadic cubes that satisfy
\begin{equation} \label{eq72}
   |Q_j|^{\frac{\alpha}{n}}\left(\dashint_{Q_j}{f^r}\right)^{\frac{1}{r}}\left(\dashint_{Q_j}{g^s}\right)^{\frac{1}{s}}>\lambda.
\end{equation}

From \eqref{eq71} and \eqref{eq72} we have
\begin{equation*}
\begin{split}
    u\left(E_{\lambda}\right) & = \sum_j\int_{Q_j}u= \sum_j|Q_j|\thinspace\dashint_{Q_j}u \\
    & \leqslant \frac{1}{\lambda^q} \sum_j{|Q_j|^{\frac{q\alpha}{n}+1}\left(\dashint_{Q_j}{u}\right)\left(\dashint_{Q_j}{f^r}\right)^{\frac{q}{r}}\left(\dashint_{Q_j}{g^s}\right)^{\frac{q}{s}}} \\
    & \leqslant \frac{1}{\lambda^q} \left[\sum_j{|Q_j|^{\frac{p\alpha}{n}+\frac{p}{q}}\left(\dashint_{Q_j}{u}\right)^{\frac{p}{q}}\left(\dashint_{Q_j}{f^r}\right)^{\frac{p}{r}}\left(\dashint_{Q_j}{g^s}\right)^{\frac{p}{s}}}\right]^{\frac{q}{p}}.
\end{split}
\end{equation*}

By using H\"older inequality for the second and the third dashed integrals and then using condition \eqref{eq21}, we obtain
\begin{equation*}
\begin{split}
    u\left(E_{\lambda}\right) & \leqslant \frac{1}{\lambda^q}\Bigg[
    \sum_j|Q_j|^{\frac{p\alpha}{n}+\frac{p}{q}-1}\left(\dashint_{Q_j}{u}\right)^{\frac{p}{q}}\left(\dashint_Q{v_1^{-\frac{r}{p_1-r}}}\right)^{\frac{p_1-r}{p_1r}p} \left(\dashint_Q{v_2^{-\frac{s}{p_2-s}}}\right)^{\frac{p_2-s}{p_2s}p} \\
    & \hspace{6.2cm}\left(\int_{Q_j}{f^{p_1}v_1}\right)^{\frac{p}{p_1}}\left(\int_{Q_j}{g^{p_2}v_2}\right)^{\frac{p}{p_2}}
    \Bigg]^{\frac{q}{p}} \\
    & \lesssim \frac{1}{\lambda^q} \left[\sum_j{\left(\int_{Q_j}{f^{p_1}v_1}\right)^{\frac{p}{p_1}}\left(\int_{Q_j}{g^{p_2}v_2}\right)^{\frac{p}{p_2}}}\right]^{\frac{q}{p}} \\
    & \leqslant \frac{1}{\lambda^q} \left(\sum_j{\int_{Q_j}{f^{p_1}v_1}}\right)^{\frac{q}{p_1}} \left(\sum_j{\int_{Q_j}{g^{p_2}v_2}}\right)^{\frac{q}{p_2}} \\
    & \leqslant \frac{1}{\lambda^q} \thinspace \|f\|_{L^{p_1}(v_1)}^q \|g\|_{L^{p_2}(v_2)}^q.
\end{split}
\end{equation*}

\hfill

\hfill

\noindent [\textit{The weak type boundedness $\Rightarrow$ condition \eqref{eq21}}]

For any cube $Q$, let $f=v_1^{-\frac{1}{p_1-r}}\chi_Q$ and $g=v_2^{-\frac{1}{p_2-s}}\chi_Q$.

If $\left(\int_Q{f^r}\right)^{\frac{1}{r}}\left(\int_Q{g^s}\right)^{\frac{1}{s}}>0$, then by choosing $\lambda=\frac{1}{2}|Q|^{\frac{\alpha}{n}}\left(\dashint_Q{f^r}\right)^{\frac{1}{r}}\left(\dashint_Q{g^s}\right)^{\frac{1}{s}}$, from the weak type boundedness we have
$$
u(Q)^{\frac{1}{q}}|Q|^{\frac{\alpha}{n}}\left(\dashint_Q{f^r}\right)^{\frac{1}{r}}\left(\dashint_Q{g^s}\right)^{\frac{1}{s}}\leqslant 2C \left(\int_{\mathbb{R}^n}f^{p_1}v_1\right)^{\frac{1}{p_1}}\left(\int_{\mathbb{R}^n}g^{p_2}v_2\right)^{\frac{1}{p_2}}.
$$

Now, if we substitute our specific choices of $f$ and $g$ into the expression, we have
$$
u(Q)^{\frac{1}{q}}|Q|^{\frac{\alpha}{n}-1}\left(\int_Q{v_1^{-\frac{r}{p_1-r}}}\right)^{\frac{1}{r}-\frac{1}{p_1}}\left(\int_Q{v_2^{-\frac{s}{p_2-s}}}\right)^{\frac{1}{s}-\frac{1}{p_2}}\leqslant 2C
$$
which is equivalent to
$$
|Q|^{\frac{\alpha}{n}+\frac{1}{q}-\frac{1}{p}}\left(\dashint_Qu\right)^{\frac{1}{q}}\left(\dashint_Q{v_1^{-\frac{r}{p_1-r}}}\right)^{\frac{p_1-r}{p_1r}}\left(\int_Q{v_2^{-\frac{s}{p_2-s}}}\right)^{\frac{p_2-s}{p_2s}}\leqslant 2C
$$
and this finishes the proof.

\hfill

%-------------------------------------------------
% NEW SECTION: PROOF OF THEOREM 2.8
%-------------------------------------------------
\section{Proof of Theorem \ref{thmH}}

As explained previously, we only need to treat the dyadic operator $\M^{r,s,\mathscr{D}}_{\alpha}$, and work with non-negative, bounded and compactly supported functions $f$ and $g$. Let $a>1$ to be chosen later. For each $k\in\mathbb{Z}$, we have
$$
\Omega_k=\{x\in\mathbb{R}^n:\thinspace \M^{r,s,\mathscr{D}}_{\alpha}(f,g)(x)>a^k\}=\bigcupdot_jQ_j^k
$$
where $Q_j^k$'s are pairwise disjoint maximal dyadic cubes satisfying
$$
   |Q_j^k|^{\frac{\alpha}{n}}\left(\dashint_{Q_j^k}{f^r}\right)^{\frac{1}{r}}\left(\dashint_{Q_j^k}{g^s}\right)^{\frac{1}{s}}>a^k.
$$

Let $E_j^k=Q_j^k\setminus\Omega_{k+1}$, then by a similar (in fact easier) argument as in the proof of theorem \ref{thmB} we have a disjoint family $\{E_j^k\}_{k,j}$ and that $|Q_j^k|\leqslant2\thinspace|E_j^k|$ with an appropriate choice of $a$. For any $h\in L^{q'}(\mathbb{R}^n)$, we have
\begin{equation*}
\begin{split}
    \int_{\mathbb{R}^n}\M^{r,s,\mathscr{D}}_{\alpha}(f,g) & (x) \thinspace h(x)\thinspace u(x)^\frac{1}{q} \thinspace dx \\
    & = \sum_{k\in \mathbb{Z}}\int_{\Omega_k\setminus\Omega_{k+1}}\M^{r,s,\mathscr{D}}_{\alpha}(f,g) (x) \thinspace h(x)\thinspace u(x)^\frac{1}{q} \thinspace dx \\
    & \leqslant \sum_{k\in \mathbb{Z}}a^{k+1}\sum_j\int_{Q_j^k}h(x)\thinspace u(x)^\frac{1}{q} \thinspace dx \\
    & \leqslant a\sum_{k,j}\left|Q_j^k\right|^{\frac{\alpha}{n}}\left(\dashint_{Q_j^k}{f^r}\right)^{\frac{1}{r}}\left(\dashint_{Q_j^k}{g^s}\right)^{\frac{1}{s}}\int_{Q_j^k}{h(x)\thinspace u(x)^\frac{1}{q} \thinspace dx} \\
    & \lesssim \sum_{k,j} \left|Q_j^k\right|^{\frac{\alpha}{n}+1} \bigl\|f^rv_1^\frac{r}{p_1}\bigr\|_{\bar{\phi}_1,Q_j^k}^{\frac{1}{r}} \bigl\|v_1^\frac{-r}{p_1}\bigr\|_{\phi_1,Q_j^k}^{\frac{1}{r}}  \bigl\|g^sv_2^\frac{s}{p_2}\bigr\|_{\bar{\phi}_2,Q_j^k}^{\frac{1}{s}} \bigl\|v_2^\frac{-s}{p_2}\bigr\|_{\phi_2,Q_j^k}^{\frac{1}{s}} \\
    & \hspace{7.4cm} \|h\|_{\bar{\psi},Q_j^k} \|u^\frac{1}{q}\|_{\psi,Q_j^k} \\
    & \lesssim \sum_{k,j} \left(\bigl\|f^rv_1^\frac{r}{p_1}\bigr\|_{\bar{\phi}_1,Q_j^k}^{\frac{1}{r}}|E_j^k|^{\frac{1}{p_1}}\right)\left(\bigl\|g^sv_2^\frac{s}{p_2}\bigr\|_{\bar{\phi}_2,Q_j^k}^{\frac{1}{s}}|E_j^k|^{\frac{1}{p_2}}\right) \\
    & \hspace{7.3cm} \left(\|h\|_{\bar{\psi},Q_j^k}|E_j^k|^{\frac{1}{q'}}\right).
\end{split}
\end{equation*}

From here, our argument will just be similar to that in \eqref{eq55}, where we will need to use the assumptions: $\bar{\psi}\in B_{q'}$, $\bar{\phi}_1\in B_{\frac{p_1}{r}}$ and $\bar{\phi}_2\in B_{\frac{p_2}{s}}$.

\hfill

%-------------------------------------------------
% NEW SECTION: PROOF OF THEOREM 2.9
%-------------------------------------------------
\section{Proof of Theorem \ref{thmI}}

When $u^\frac{1}{q}=v_1^\frac{1}{p_1}v_2^\frac{1}{p_2}$ and $\frac{1}{q}=\frac{1}{p}-\frac{\alpha}{n}$, condition \eqref{eq21} becomes
\begin{equation} \label{eq91}
    \sup_Q \left(\dashint_Q{v_1^\frac{q}{p_1}v_2^\frac{q}{p_2}}\right)^{\frac{1}{q}} \left(\dashint_Q{v_1^{-\frac{r}{p_1-r}}}\right)^{\frac{p_1-r}{p_1r}} \left(\dashint_Q{v_2^{-\frac{s}{p_2-s}}}\right)^{\frac{p_2-s}{p_2s}} < \infty
\end{equation}
which implies
$$
u \in A_{2q} \quad \text{ and } \quad v_1^{-\frac{r}{p_1-r}} \in A_{\frac{2p_1r}{p_1-r}} \quad \text{ and } \quad v_2^{-\frac{s}{p_2-s}} \in A_{\frac{2p_2s}{p_2-s}}
$$
by using theorem \ref{thm5}. 
Then by theorem \ref{thm10}, there exists $m>1$ such that

\begin{equation*}
\begin{split}
    \left(\dashint_Q{u^m}\right)^{\frac{1}{mq}} & \leqslant \left(\dashint_Q{u}\right)^{\frac{1}{q}} \\
    \left(\dashint_Q{v_1^{-\frac{mr}{p_1-r}}}\right)^{\frac{p_1-r}{mp_1r}} & \leqslant \left(\dashint_Q{v_1^{-\frac{r}{p_1-r}}}\right)^{\frac{p_1-r}{p_1r}} \\
    \left(\dashint_Q{v_2^{-\frac{ms}{p_2-s}}}\right)^{\frac{p_2-s}{mp_2s}} & \leqslant \left(\dashint_Q{v_2^{-\frac{s}{p_2-s}}}\right)^{\frac{p_2-s}{p_2s}}.
\end{split}
\end{equation*}

These inequalities together with \eqref{eq91} imply
\begin{equation} \label{eq92}
    \sup_Q \left(\dashint_Q{u^m}\right)^{\frac{1}{mq}} \left(\dashint_Q{v_1^{-\frac{mr}{p_1-r}}}\right)^{\frac{p_1-r}{mp_1r}} \left(\dashint_Q{v_2^{-\frac{ms}{p_2-s}}}\right)^{\frac{p_2-s}{mp_2s}} < \infty.
\end{equation}

Now, if we consider the Young functions: $\psi(t)=t^{mq}$, $\phi_1(t)=t^{\frac{mp_1}{p_1-r}}$ and $\phi_2(t)=t^{\frac{mp_2}{p_2-s}}$, then we have $\bar{\psi} \in B_{q'}$, $\bar{\phi}_1 \in B_{\frac{p_1}{r}}$ and $\bar{\phi}_2 \in B_{\frac{p_2}{s}}$. Moreover, we can reformulate \eqref{eq92} as
\begin{equation*}
    \sup_Q |Q|^{\frac{\alpha}{n}+\frac{1}{q}-\frac{1}{p}}\bigl\|u^\frac{1}{q}\bigr\|_{\psi,Q}\thinspace\bigl\|v_1^\frac{-r}{p_1}\bigr\|_{\phi_1,Q}^{\frac{1}{r}}\thinspace\bigl\|v_2^\frac{-s}{p_2}\bigr\|_{\phi_2,Q}^{\frac{1}{s}} < \infty
\end{equation*}
where we used the Sobolev condition $\frac{\alpha}{n}+\frac{1}{q}-\frac{1}{p}=0$. This is exactly the condition on the weights $(u,v_1,v_2)$ in Theorem \ref{thmE}, so the conclusion is immediate.

\hfill

%-------------------------------------------------
% NEW SECTION: A BILINEAR STEIN-WEISS INEQUALITY
%-------------------------------------------------
\section{Applications and examples}

In \cite{SW} Stein and Weiss proved the following inequality:
$$\int_{\R^n}\int_{\R^n}\frac{f(x)g(y)}{|x|^\gamma|x-y|^\al|y|^\beta}\,dxdy \lesssim \|f\|_p\|g\|_{q'}$$
where $\al$, $\beta$, and $\gamma$ are positive numbers that depend on $p$ and $q$.  
Below we have a bilinear Stein-Weiss inequality for the case when $1<p\leqslant q<\infty$. The case when $\frac{1}{2}<p\leqslant q\leqslant 1$ was done by the second author \cite{K. Moen bilinear}.

\begin{theorem} \label{bilinear Stein-Weiss}
Suppose $1<p_1,p_2<\infty$ and $1<p\leqslant q<\infty$. If $\alpha,\beta,\gamma_1,\gamma_2$ satisfy
\begin{equation} \label{eq101}
    \beta <\frac{n}{q}, \hspace{2mm} \gamma_1<(p-1)\frac{n}{p_1}, \hspace{2mm} \gamma_2<(p-1)\frac{n}{p_2}
\end{equation}
\begin{equation}
    \alpha+\beta+\gamma_1+\gamma_2=n+\frac{n}{q}-\frac{n}{p}
\end{equation}
\begin{equation} \label{eq103}
    \beta+\gamma_1+\gamma_2\geqslant0
\end{equation}
Then we have
\begin{equation} \label{eq104}
   \int_{\mathbb{R}^n}\int_{\mathbb{R}^n}\frac{f(x-y) \thinspace g(x+y) \thinspace h(x)}{|y|^\al \thinspace|x-y|^{\gamma_1} \thinspace|x+y|^{\gamma_2} \thinspace|x|^\beta} \thinspace dxdy \lesssim \|f\|_{p_1}  \thinspace \|g\|_{p_2}  \thinspace \|h\|_{q'}
\end{equation}
for non-negative functions $f,g$, and $h$.  
\end{theorem}

\begin{remark}
Condition \eqref{eq101} corresponds to the condition
$$
\beta <(1-q)\frac{n}{q}, \hspace{2mm} \gamma_1<\frac{n}{p_1}, \hspace{2mm} \gamma_2<\frac{n}{p_2}
$$
which is stated in \cite{K. Moen bilinear}. The interesting phenomenon here is that the factor $1-q$ (for the case $p\leqslant q\leqslant 1$) has become $p-1$ (for the case $1<p\leqslant q$). This may reveal some clues about the case $p\leqslant1<q$.
\end{remark}

\begin{remark}
If we think of the linear case just as a restriction of the bilinear one, then we can just drop $p_2$ and $\gamma_2$, and identify $p_1$ with $p$, $\gamma_1$ with $\gamma$. At that time, conditions \eqref{eq101}-\eqref{eq103} will become
\begin{equation*}
    \beta <\frac{n}{q}, \hspace{2mm} \gamma_1<\frac{n}{p'}
\end{equation*}
\begin{equation*}
    \alpha+\beta+\gamma=n+\frac{n}{q}-\frac{n}{p}
\end{equation*}
\begin{equation*}
    \beta+\gamma\geqslant0
\end{equation*}
which are exactly the needed conditions for the (linear) Stein-Weiss inequality to hold true.
\end{remark}

\begin{remark}
The inequality \eqref{eq104} can also be written as 
\begin{equation*}
   \int_{\mathbb{R}^n}\int_{\mathbb{R}^n}\frac{f(x) \thinspace g(y) \thinspace h\big(\frac{x+y}{2})}{|x-y|^\al \thinspace|x|^{\gamma_1} \thinspace|y|^{\gamma_2} \thinspace|x+y|^\beta} \thinspace dxdy \lesssim \|f\|_{p_1}  \thinspace \|g\|_{p_2}  \thinspace \|h\|_{q'}.
\end{equation*}
\end{remark}

\begin{proof} [Proof of Theorem \ref{bilinear Stein-Weiss}]
Inequality \eqref{eq104} is just a dualized form of the following inequality with some appropriate weight-scalings on the functions $f$ and $g$.

\begin{equation*}
\begin{split}
    & \left[\int_{\mathbb{R}^n}\Bigl(\bigl|\mathsf{BI}_{n-\alpha}(f,g)(x)\bigr||x|^{-\beta}\Bigr)^qdx\right]^\frac{1}{q} \\
    & \hspace{2cm} \lesssim \left[\int_{\mathbb{R}^n}\Bigl(|f(x)||x|^{\gamma_1}\Bigr)^{p_1}dx\right]^\frac{1}{p_1} \left[\int_{\mathbb{R}^n}\Bigl(|g(x)||x|^{\gamma_2}\Bigr)^{p_2}dx\right]^\frac{1}{p_2}.
\end{split}
\end{equation*}

We are going to apply theorem \ref{thmE} here, so we only need to check condition \eqref{eq21} with $u=|x|^{-q\beta}$, $v_1=|x|^{p_1\gamma_1}$, $v_2=|x|^{p_2\gamma_2}$, $r=\frac{p_1}{p}$ and $s=\frac{p_2}{p}$. To be clearer, we need to show that
\begin{equation} \label{eq105}
    \sup_Q |Q|^{\frac{n-\alpha}{n}+\frac{1}{q}-\frac{1}{p}}\left(\dashint_Q{|x|^{-\beta q}}\right)^{\frac{1}{q}} \left(\dashint_Q{|x|^{\frac{\gamma_1p_1}{p-1}}}\right)^{\frac{p-1}{p_1}} \left(\dashint_Q{|x|^{\frac{\gamma_2p_2}{p-1}}}\right)^{\frac{p-1}{p_2}} < \infty.
\end{equation}

\hfill

Now, for any cube $Q$, let $Q_0=Q\bigl(\text{O},\ell(Q)\bigr)$. We then either have $Q\cap Q_0=\emptyset$ or $Q\cap Q_0\neq\emptyset$. If $Q\cap Q_0=\emptyset$, then $|x|\sim|x|_\infty\geqslant\ell(Q)$ for all $x\in Q$. This implies that the left hand side of \eqref{eq105} is bounded by

$$
 \sup_Q  |\ell(Q)|^{n-\alpha+\frac{n}{q}-\frac{n}{p}}\thinspace|\ell(Q)|^{-\beta-\gamma_1-\gamma_2}=1.
$$

\hfill

If $Q\cap Q_0\neq\emptyset$, then $|x|\leqslant\sqrt{n}\thinspace|x|_\infty\leqslant2\sqrt{n}\thinspace\ell(Q)$ for all $x\in Q$. This implies that $Q\subset B=B\bigl(\text{O},2\sqrt{n}\thinspace\ell(Q)\bigr)$, and hence the left hand side of \eqref{eq105} is bounded by a constant. 
\end{proof}

Finally we end the paper with an example to show that our condition, condition \eqref{onevecwp=q}, on the weights for $\mathsf{BM}$ is more general than the known results $(w_1,w_2)\in A_p\times A_p$.  In fact we will show that there exists weights $(w_1,w_2)$ that satisfy \eqref{onevecwp=q} but $w_1\notin A_p$ and $w_2\notin A_p$. Here, we are going to give an example of weights $w_1$, and $w_2$ for this fact. Consider $w_1=|x|^\alpha$, and $w_2=|x|^\beta$. We shall prove

$$
K = \sup_Q\left(\dashint_Q|x|^{\frac{p\alpha}{p_1}+\frac{p\beta}{p_2}}\right)^\frac{1}{p} \left(\dashint_Q|x|^\frac{\alpha}{1-p}\right)^\frac{p-1}{p_1} \left(\dashint_Q|x|^\frac{\beta}{1-p}\right)^\frac{p-1}{p_2} < \infty.
$$

\hfill

For every cube $Q$, we have 2 situations: either $|c_Q|_\infty\leqslant2\ell(Q)$ or $|c_Q|_\infty>2\ell(Q)$.

If $|c_Q|_\infty\leqslant2\ell(Q)$, then
\begin{equation*}
\begin{split}
    K & \leqslant \ell(Q)^{-n} \left(\int_{B_0}|x|^{\frac{p\alpha}{p_1}+\frac{p\beta}{p_2}}\right)^\frac{1}{p}  \left(\int_{B_0}|x|^\frac{\alpha}{1-p}\right)^\frac{p-1}{p_1} \left(\int_{B_0}|x|^\frac{\beta}{1-p}\right)^\frac{p-1}{p_2} \\
    & \approx \ell(Q)^{-n+\frac{\alpha}{p_1}+\frac{\beta}{p_2}+\frac{n}{p}+\frac{n(p-1)-\alpha}{p_1}+\frac{n(p-1)-\beta}{p_2}} = 1
\end{split}
\end{equation*}
where $B_0=B\bigl(3\sqrt{n}\thinspace\ell(Q)\bigr)$, and whenever $\alpha<n(p-1)$, $\beta<n(p-1)$, $-n<\frac{p\alpha}{p_1}+\frac{p\beta}{p_2}$.

If $|c_Q|_\infty>2\ell(Q)$, then $|x|\sim|x|_\infty\sim|c_Q|_\infty\sim|c_Q|$ and hence
\begin{equation*}
    K \approx |c_Q|^{\frac{\alpha}{p_1}+\frac{\beta}{p_2}-\frac{\alpha}{p_1}-\frac{\beta}{p_2}} = 1.
\end{equation*}

These mean that $K<\infty$ whenever $\alpha<n(p-1)$, $\beta<n(p-1)$, $-n<\frac{p\alpha}{p_1}+\frac{p\beta}{p_2}$. So, we may have $\alpha$ get close to $-n(1+p_1-p)$, which is less than $-n$, as long as $\beta<n(p-1)$. Similarly, we may have $\beta$ get close to $-n(1+p_2-p)$, which is less than $-n$, as long as $\alpha<n(p-1)$. This fact provides a wider range for $\alpha$ and $\beta$ because the $A_p\times A_p$ requires $-n<\alpha,\beta<n(1-p)$.

%-------------------------------------------------
% NEW SECTION: AN APPLICATION OF OUR MAXIMAL OPERATOR
%-------------------------------------------------

%-------------------------------------------------
% REFERENCES
%-------------------------------------------------
\bibliographystyle{plain}

\begin{thebibliography}{99}

%\bibitem {D. Adams} D. Adams, \textit{A note on Riesz potentials}, Duke Math. J. \textbf{42} (1975), 765-778.

\bibitem{S. Chanillo} S. Chanillo, \textit{A note on commutators}, Indiana Univ. Math. J. \textbf{31(1)} (1982), 7-16.

%\bibitem {D. Chung} D. Chung, \textit{Sharp estimates for the commutators of the Hilbert, Riesz transforms and the Beurling-Ahlfors operator on weighted Lebesgue spaces}, arXiv:1001.0755 [math.FA] (2010).

%\bibitem {D. Chung and M. C. Pereyra and C. Perez} D. Chung, M. C. Pereyra and C. P\'erez, \textit{Sharp bounds for general commutators on weighted Lebesgue spaces}, Trans. Amer. Math. Soc. (2010).

%\bibitem {R. R. Coifman and Peter W. Jones and Stephen Semmes} R. R. Coifman, Peter W. Jones and Stephen Semmes, \textit{Two elementary proofs of the $L^2$ boundedness of Cauchy integrals on Lipschitz curves}, Amer. Math. Soc. \textbf{2} (1989), 553-564.

\bibitem {R. R. Coifman and R. Rochberg and G. Weiss} R. R. Coifman, R. Rochberg and G. Weiss, \textit{Factorization theorems for Hardy spaces in several variables}, Ann. of Math. \textbf{103(3)} (1976), 611-635.

%\bibitem {D. Cruz-Uribe and K. Moen a fractional} D. Cruz-Uribe and K. Moen, \textit{A Fractional Muckenhoupt-Wheeden Theorem and its Consequences}, Integr. Equ. Oper. Theory \textbf{76} (2013), 421-446.

%\bibitem {D. Cruz-Uribe and K. Moen} D. Cruz-Uribe and K. Moen, \textit{One and two weight norm inequalities for Riesz potentials}, Illinois Journal of Mathematics \textbf{57} (2013), 295-323.

\bibitem {D. Cruz-Uribe and K. Moen sharp norm for commutators} D. Cruz-Uribe and K. Moen, \textit{Sharp norm inequalities for commutators of classical operators}, Publ. Mat \textbf{56} (2012), 147-190.

\bibitem {D. Cruz-Uribe and J. M. Martell and C. Perez} D. Cruz-Uribe, J. M. Martell and C. P\'erez, \textit{Weights, extrapolation and the theory of Rubio de Francia}, Operator Theory: Advances and Applications, 215, Birkhauser, Basel, (2011).

%\bibitem {Y. Ding and C. Lin} Y. Ding and C. Lin, \textit{Rough bilinear fractional integrals}, Math. Nachr. \textbf{246-247} (1975), 47-52.

\bibitem {J. Duoandikoetxea} J. Duoandikoetxea, \textit{Fourier Analysis}, Translated by D. Cruz-Uribe, Amer. Math. Soc., Providence, RI, 2000.

%\bibitem {G. B. Folland} G. B. Folland, \textit{Real Analysis: Modern Techniques and Their Applications}, 2nd ed., Wiley, 1999.

\bibitem{L. Grafakos On multi} L. Grafakos, \textit{On multilinear fractional integrals}, Studia Math. \textbf{102} (1992), 49-56.

\bibitem {L. Grafakos} L. Grafakos, \textit{Modern Fourier Analysis}, Springer, LLC, 2009.

\bibitem {J. L. Journe} J. L. Journ\'e, \textit{Calder\'on-Zygmund operators, pseudo differential operators and the Cauchy integral of Calder\'on}, Lecture Notes in Mathematics \textbf{994} (1983), Springer-Verlag, Berlin.

\bibitem {C. E. Kenig and E. M. Stein} C. E. Kenig and E. M. Stein, \textit{Multilinear estimates and fractional integration}, Math. Res. Lett. \textbf{6} (1999), 1-15.

\bibitem {M. A. Kranoselskii and J. B. Rutickii} M. A. Kranosel'ski\u{i} and J. B. Ruticki\u{i}, \textit{Convex functions and Orlicz spaces}. Translated from the first Russian edition by Leo F. Boron. P. Noordhoff Ltd., Groningen (1961).

\bibitem {M. Lacey} M. Lacey, \textit{The bilinear maximal functions map into $L^p$ for $2/3<p\leq 1$}, Ann. of Math. \textbf{151(1)} (2000), 35--57.

\bibitem {A. K. Lerner} A. K. Lerner, \textit{On an estimate of Calder\'on-Zygmund operators by dyadic positive operators}, J. Anal. Math. \textbf{121} (2013), 141-161.

\bibitem {K. Moen multilinear} K. Moen, \textit{Weighted inequalities for multilinear fractional integral operators}, Coll. Math. \textbf{60} (2009), 213-238.

%\bibitem {K. Moen sharp} K. Moen, \textit{Sharp one-weight and two-weight bounds for maximal operators}, Studia Math. \textbf{194} (2009), 163-180.

\bibitem {K. Moen bilinear} K. Moen, \textit{New weighted estimates for bilinear fractional integral operators}, Trans. Amer. Math. Soc.
\textbf{366} (2014), 627-646.

\bibitem {B. Muckenhoupt and R.L. Wheeden} B. Muckenhoupt and R.L. Wheeden, \textit{Weighted norm inequalities for fractional integrals}, Trans. Amer. Math. Soc. \textbf{192} (1974), 261-274.

\bibitem{CP} C. P\'erez, \textit{Two weighted inequalities for Potential and Fractional Type Maximal Operators}, Indiana Univ. Math. J. \textbf{43} (1994), 1-28.

\bibitem {C. Perez} C. P\'erez, \textit{On sufficient conditions for the boundedness of the Hardy-Littlewood maximal operator between weighted $L^p$-spaces with different weights}, Proc. London Math. Soc. \textbf{71} (1995), 135-157.



%\bibitem {C. Perez and Gladis Pradolini and Rodolfo H. Torres and Rodrigo Trujillo-Gonzalez} C. P\'erez, Gladis Pradolini, Rodolfo H. Torres and Rodrigo Trujillo-Gonz\'alez, \textit{End-Point Estimates for iterated commutators of multilinear singular integrals}, Bulletin of the London Mathematical Society \textbf{46} (2014), 26-44.

\bibitem {C. Perez and Israel P. Rivera-Rios} C. P\'erez and Israel P. Rivera-R\'ios, \textit{Borderline weighted estimates for commutators of singular integrals}, arXiv:1507.08568 [math.CA] (2015).

\bibitem{SW} E. Stein and G. Weiss, \textit{Fractional integrals on $n$-dimensional Euclidean space}, J. Math. Mech. {\bf 7} (1958), 503--514.

%\bibitem {G.V. Welland} G.V. Welland$^1$, \textit{Weighted norm inequalities for fractional integrals}, Trans. Amer. Math. Soc. \textbf{51} (1975), 143-148.

\end{thebibliography}

\hfill

\end{document}